\documentclass[12pt,  ps2pdf]{article}

\usepackage{graphicx}

\usepackage{pgf}
\usepackage{tikz}
\usetikzlibrary{arrows}
\usetikzlibrary{shapes, backgrounds}
\usepackage[absolute,overlay]{textpos}
\usepackage{fancybox}
\usepackage{color}

\newcommand{\beq}{\begin{eqnarray*}}
\newcommand{\eeq}{\end{eqnarray*}}

\newcommand{\bet}{\begin{tikzpicture}}
\newcommand{\ent}{\end{tikzpicture}}

\makeatletter
\renewcommand{\theequation}{\thesection.\arabic{equation}}
\@addtoreset{equation}{section}

\@addtoreset{figure}{section}
\def\eqnarray{%
\stepcounter{equation}%
\let\@currentlabel=\theequation
\global\@eqnswtrue
\global\@eqcnt\z@
\tabskip\@centering
\let\\=\@eqncr
$$\halign to \displaywidth\bgroup\@eqnsel\hskip\@centering
$\displaystyle\tabskip\z@{##}$&\global\@eqcnt\@ne
\hfil$\displaystyle{{}##{}}$\hfil
&\global\@eqcnt\tw@$\displaystyle\tabskip\z@{##}$\hfil
\tabskip\@centering&\llap{##}\tabskip\z@\cr}
\makeatother

\newtheorem{theorem}{Theorem}[section]
\newtheorem{lemma}[theorem]{Lemma}

\newtheorem{proposition}[theorem]{Proposition}
\newtheorem{remark}[theorem]{Remark}

\newsavebox{\toy}
\savebox{\toy}{\framebox[0.65em]{\rule{0cm}{1ex}}}
\newcommand{\QED}{\usebox{\toy}}
\def\nlni{\par\ifvmode\removelastskip\fi\vskip\baselineskip\noindent}
\newenvironment{proof}{\nlni\begingroup\it Proof.\rm}{
\endgroup\vskip\baselineskip}

\begin{document}
\setlength{\baselineskip}{15pt}
\title{
Hitting Matrix and Domino Tiling with Diagonal Impurities : 
}
\author{
Fumihiko Nakano
\thanks{
Department of Mathematics,
Gakushuin University,
1-5-1, Mejiro, Toshima-ku, Tokyo, 171-8588, Japan.
e-mail : 
fumihiko@math.gakushuin.ac.jp}
\and 
Taizo Sadahiro
\thanks{Department of Computer Science, 
Tsuda Colledge, Tokyo, Japan.
e-mail : sadahiro@tsuda.ac.jp}
}

\maketitle

\begin{abstract}
As a continuation to our previous work
\cite{NS1, NS2}, 
we consider the domino tiling problem with impurities. 
(1)
if
we have more than two impurities on the boundary, 
we can compute the number of corresponding perfect matchings by using the hitting matrix method\cite{Fomin}.
(2)
we have 
an alternative proof of the main result in \cite{NS1} and result in (1) above using the formula by Kenyon-Wilson \cite{KW1, KW2} of counting the number of groves on the circular planar graph.
(3)
we study 
the behavior of the probability of finding the impurity at a given site when the size of the graph tends to infinity, as well as the scaling limit of those.
\end{abstract}


\section{Introduction}
\subsection{Background}
Let
$G = (V(G), E(G))$
be a graph.
A subset 
$M$ 
of 
$E(G)$
is called a 
{\bf perfect matching}(or a {\bf dimer covering}) on 
$G$ 
if and only if 
for any
$x \in V(G)$
there exists 
$e \in M$
uniquely with 
$x \in e$.
Let 
${\cal M}(G)$ 
be the set of perfect matchings on 
$G$. 
The perfect matching problem 
was first studied by  
Kasteleyn, Temperley-Fisher, 
in the context of the statistical mechanics, and many papers appeared since then. 
In particular, 
for the perfect matching on the bipartite graphs, such as domino or lozenge tilings, many results have been known
(e.g., \cite{K} and references therein).

In this paper 
we consider the perfect matching on a non-biparite graph
$G^{(k)}$ 
defined below. 
We use the same notation in 
\cite{NS2}.
Let 
$G_2$
be a subgraph of the square lattice, and let 
$G_1$
be its dual graph(Figure \ref{G2}). 
%

\begin{figure}
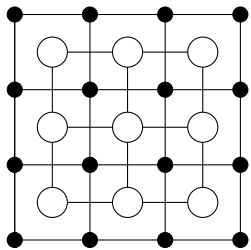

\begin{center}
\bet

\draw (0,0) --++(0,3)  ; 

\foreach \x in {0, 1, 2, 3}
\draw (\x, 0)  --++(0, 3); 

\foreach \y in {0, 1, 2, 3}
\draw (0, \y) --++(3, 0) ; 

\foreach \x in {0, 1, 2, 3}
 \foreach \y in {0, 1, 2, 3}
  {
   \draw (\x, \y) [ fill = black ] circle ( 0.1 cm ) ; 
   }

\begin{scope} [ xshift = 0.5 cm, yshift = 0.5 cm ] 
{

\draw (0, 0) grid +(2, 2) ; 

\foreach \x in {0, 1, 2}
\foreach \y in {0, 1, 2}
{
\draw [ fill = white ]  (\x, \y) circle (0.2 cm ) ; 
}

}
\end{scope}

\ent

\end{center}
\caption{
$G_2$
and 
$G_1$
}
\label{G2}
\end{figure}
Let 
$G_{1, T}$
be the graph made by adding 
$(2k-1)$
vertices
$T_1, T_2, \cdots, T_{2k-1}$,
which we call the {\bf terminal}, 
to the boundary 
$\partial G_1$ 
of 
$G_1$
(Figure \ref{G1}).
%

\begin{figure}
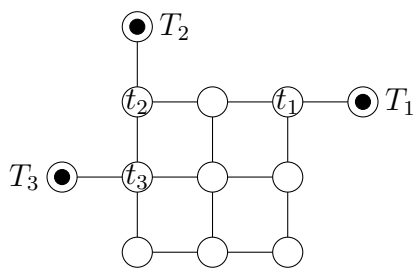

\begin{center}
\bet

\draw (0,0) grid +(2, 2) ; 

\draw (0, 2) --+(0, 1) ; 
\draw (0, 1) --+(-1, 0) ; 
\draw (2, 2) --+(1, 0) ; 

\foreach \x in {0, 1, 2}
\foreach \y in {0, 1, 2}
{
\draw (\x, \y) [ fill = white ] circle (0.2 cm ) ; 
}

\draw (2, 2) node {$t_1$} ; 
\draw (0, 2) node {$t_2$} ; 
\draw (0, 1) node {$t_3$} ; 

\draw (0, 3) [ fill = white ]  circle ( 0.2 cm ) ; 
\draw (0, 3) [ fill = black ] circle ( 0.1 cm ) ; 

\draw (-1, 1) [ fill = white ]  circle ( 0.2 cm ) ; 
\draw (-1, 1) [ fill = black ] circle ( 0.1 cm ) ; 

\draw (3, 2) [ fill = white ]  circle ( 0.2 cm ) ; 
\draw (3, 2) [ fill = black ] circle ( 0.1 cm ) ; 

\draw (3.5, 2) node {$T_1$}; 
\draw (0.5, 3) node {$T_2$}; 
\draw (-1.5, 1) node {$T_3$}; 

\ent
\end{center}
\caption{
An example of 
$G_{1, T}$
for
$k=2$. 
}
\label{G1}
\end{figure}
We next superimpose 
$G_2$
and 
$G_{1, T}$, 
and we make a vertex wherever two edges cross, and call it a {\bf middle vertex}.
For the vertices 
in the boundary of 
$G_1$, 
we add extra edges toward the outer face of 
$G_2$, 
and make vertices, which we call the {\bf boundary vertex}, wherever two edges meet(Figure \ref{Superimpose}). 
%

\begin{figure}
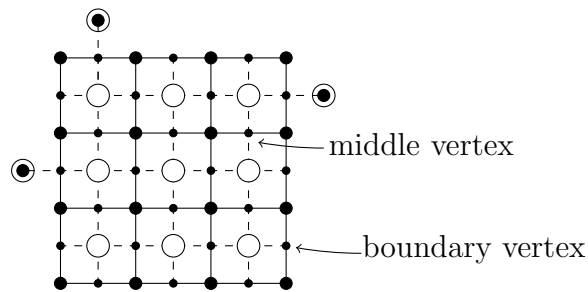

\begin{center}
\bet 

\draw (0, 0) grid +(3, 3) ; 

\begin{scope} [ xshift = 0.5cm, yshift = 0.5 cm ] 
{

\draw [ dashed  ] (0, 0) grid +(2, 2) ; 

}
\end{scope}

\foreach \x in {0, 1, 2, 3}
\foreach \y in {0, 1, 2, 3}
{
\draw (\x, \y) [ fill = black ] circle ( 0.08 cm ) ; 
}

\begin{scope} [ xshift = 0.5 cm, yshift = 0.5 cm ]
{
\draw (0, 3) [ fill = white ] circle ( 0.15 cm ) ; 
\draw (0, 3) [ fill = black ] circle ( 0.08 cm ) ; 

\draw (-1, 1)  [ fill = white ] circle ( 0.15 cm ) ; 
\draw (-1, 1) [ fill = black ] circle ( 0.08 cm ) ; 

\draw (3, 2) [ fill = white ]  circle ( 0.15 cm ) ; 
\draw (3, 2) [ fill = black ] circle ( 0.08 cm ) ; 

\draw [ dashed ]  (0, 2) --+(0, 1) ; 
\draw [ dashed ]  (0, 1) --+(-1, 0) ; 
\draw [ dashed ]  (2, 2) --+(1, 0) ; 

\draw [ dashed ] (0, 0) --+(-0.5, 0) ; 
\draw [ dashed ] (0, 2) --+(-0.5, 0) ; 
\draw [ dashed ] (2, 0) --+(0.5, 0) ; 
\draw [ dashed ] (2, 1) --+(0.5, 0) ; 

\foreach \x in {0, 1, 2}
{
\draw [ dashed ] (\x, 0) --+(0, -0.5) ; 
\draw [ dashed ] (\x, 2) --+(0, 0.5) ; 
}

}
\end{scope}

\foreach \x in {0, 0.5, 1, 1.5, 2, 2.5,  3}
\foreach \y in {0, 0.5, 1, 1.5, 2, 2.5,  3}
{
\draw (\x, \y) [ fill = black ] circle (0.05 cm) ; 
}

\foreach \x in {0, 1, 2}
\foreach \y in {0, 1, 2}
{
\draw ( \x+0.5, \y+0.5 ) [ fill = white ] circle ( 0.15 cm ) ; 
}

\draw [ -> ] (3.5, 1.8) arc (270 : 260 : 5) ; 
\draw (4.8, 1.85) node {middle vertex} ; 

\draw [ -> ] (4, 0.4) arc (270 : 260 : 5) ; 
\draw (5.5, 0.45) node {boundary vertex} ; 

\ent
\end{center}
\caption{
We superimpose  $G_2$ and $G_{1, T}$ 
and then add middle and boundary vertices
}
\label{Superimpose}
\end{figure}
By adding 
furthermore diagonal edges alternately, we have the graph
$G^{(k)}$(Figure \ref{OurGraph}). 
%

\begin{figure}
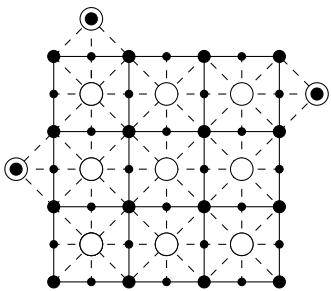

\begin{center}
\bet 

\draw (0, 0) grid +(3, 3) ; 

\begin{scope} [ xshift = 0.5cm, yshift = 0.5 cm ] 
{
\draw [ dashed  ] (0, 0) grid +(2, 2) ; 
}
\end{scope}

\foreach \x in {0, 1, 2, 3}
\foreach \y in {0, 1, 2, 3}
{
\draw (\x, \y) [ fill = black ] circle ( 0.08 cm ) ; 
}

\begin{scope} [ xshift = 0.5 cm, yshift = 0.5 cm]
{

\draw [ dashed ]  (0, 2) --+(0, 1) ; 
\draw [ dashed ]  (0, 1) --+(-1, 0) ; 
\draw [ dashed ]  (2, 2) --+(1, 0) ; 

\draw [ dashed ] (0, 0) --+(-0.5, 0) ; 
\draw [ dashed ] (0, 2) --+(-0.5, 0) ; 
\draw [ dashed ] (2, 0) --+(0.5, 0) ; 
\draw [ dashed ] (2, 1) --+(0.5, 0) ; 

\foreach \x in {0, 1, 2}
{
\draw [ dashed ] (\x, 0) --+(0, -0.5) ; 
\draw [ dashed ] (\x, 2) --+(0, 0.5) ; 
}

}
\end{scope}

\foreach \x in {0, 0.5, 1, 1.5, 2, 2.5,  3}
\foreach \y in {0, 0.5, 1, 1.5, 2, 2.5,  3}
{
\draw (\x, \y) [ fill = black ] circle (0.05 cm) ; 
}

\foreach \y in {0, 1, 2} 
{
\draw [ dashed ] (0, \y)  --++(0.5, 0.5) --++(0.5, -0.5)
 --++(0.5, 0.5) --++(0.5, -0.5) --++(0.5, 0.5) --++(0.5, -0.5) ; 
\draw [ dashed ] (0, \y+1) --++(0.5, -0.5) --++(0.5, 0.5)
--++(0.5, -0.5) --++(0.5, 0.5) --++(0.5, -0.5) --++(0.5, 0.5) ;  
}

\draw [ dashed ] (0, 1) --++(-0.5, 0.5) --++(0.5, 0.5) ; 
\draw [ dashed ] (0, 3) --++(0.5, 0.5) --++(0.5, -0.5) ; 
\draw [ dashed ] (3, 2) --++(0.5, 0.5) --++(-0.5, 0.5) ; 

\foreach \x in {0, 1, 2, }
\foreach \y in {0, 1, 2, }
{
\draw (\x+0.5, \y+0.5) [ fill = white ] circle ( 0.15 cm ) ; 
}

\begin{scope} [ xshift = 0.5 cm, yshift = 0.5 cm ] 
{
\draw (0, 3) [ fill = white ] circle ( 0.15 cm ) ; 
\draw (0, 3) [ fill = black ] circle ( 0.08 cm ) ; 

\draw (-1, 1) [ fill = white ] circle ( 0.15 cm ) ; 
\draw (-1, 1) [ fill = black ] circle ( 0.08 cm ) ; 

\draw (3, 2) [ fill = white ] circle ( 0.15 cm ) ; 
\draw (3, 2) [ fill = black ] circle ( 0.08 cm ) ; 
}
\end{scope}

\ent
\end{center}
\caption{
Graph $G^{(2)}$
}
\label{OurGraph}
\end{figure}
Figure \ref{Dimer-Example}
shows an example of a perfect matching on 
$G^{(2)}$, 
where we have two kinds of edges. 
For 
$M \in {\cal M}(G)$, 
we say an edge 
$e \in M$ 
is an {\bf impurity} if 
it is an edge between the vertices of 
$G_{1,T}$ 
and 
$G_2$.
The number of impurities is constant on 
${\cal M}(G^{(k)})$ 
and is equal to 
$k$. 
%

\begin{figure}
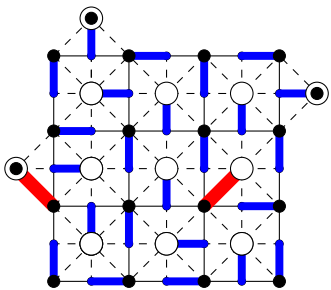

\begin{center}
\bet 

\draw (0, 0) grid +(3, 3) ; 

\begin{scope} [ xshift = 0.5cm, yshift = 0.5 cm ] 
{
\draw [ dashed  ] (0, 0) grid +(2, 2) ; 
}
\end{scope}

\begin{scope} [ xshift = 0.5 cm, yshift = 0.5 cm]
{

\draw [ dashed ]  (0, 2) --+(0, 1) ; 
\draw [ dashed ]  (0, 1) --+(-1, 0) ; 
\draw [ dashed ]  (2, 2) --+(1, 0) ; 

\draw [ dashed ] (0, 0) --+(-0.5, 0) ; 
\draw [ dashed ] (0, 2) --+(-0.5, 0) ; 
\draw [ dashed ] (2, 0) --+(0.5, 0) ; 
\draw [ dashed ] (2, 1) --+(0.5, 0) ; 

\foreach \x in {0, 1, 2}
{
\draw [ dashed ] (\x, 0) --+(0, -0.5) ; 
\draw [ dashed ] (\x, 2) --+(0, 0.5) ; 
}

}
\end{scope}

\foreach \x in {0, 0.5, 1, 1.5, 2, 2.5,  3}
\foreach \y in {0, 0.5, 1, 1.5, 2, 2.5,  3}
{
\draw (\x, \y) [ fill = black ] circle (0.05 cm) ; 
}

\foreach \y in {0, 1, 2} 
{
\draw [ dashed ] (0, \y)  --++(0.5, 0.5) --++(0.5, -0.5)
 --++(0.5, 0.5) --++(0.5, -0.5) --++(0.5, 0.5) --++(0.5, -0.5) ; 
\draw [ dashed ] (0, \y+1) --++(0.5, -0.5) --++(0.5, 0.5)
--++(0.5, -0.5) --++(0.5, 0.5) --++(0.5, -0.5) --++(0.5, 0.5) ;  
}

\draw [ dashed ] (0, 1) --++(-0.5, 0.5) --++(0.5, 0.5) ; 
\draw [ dashed ] (0, 3) --++(0.5, 0.5) --++(0.5, -0.5) ; 
\draw [ dashed ] (3, 2) --++(0.5, 0.5) --++(-0.5, 0.5) ; 

\draw [ line cap = round, line width = 3, color = blue ] (0,0) --+(0, 0.5) ; 
\draw [ line cap = round, line width = 3, color = blue ] (0.5,0) --+(0.5, 0) ; 
\draw [ line cap = round, line width = 3, color = blue ] (1.5,0) --+(0.5, 0) ; 
\draw [ line cap = round, line width = 3, color = blue ] (2.5,0) --+(0, 0.5) ; 
\draw [ line cap = round, line width = 3, color = blue ] (3,0) --+(0, 0.5) ; 
\draw [ line cap = round, line width = 3, color = blue ] (0.5,0.5) --+(0, 0.5) ; 
\draw [ line cap = round, line width = 3, color = blue ] (1,0.5) --+(0, 0.5) ; 
\draw [ line cap = round, line width = 3, color = blue ] (1.5,0.5) --+(0.5, 0) ; 
\draw [ line cap = round, line width = 3, color = blue ] (1.5,1) --+(0, 0.5) ; 
\draw [ line cap = round, line width = 3, color = blue ] (2.5,1) --+(0.5, 0) ; 
\draw [ line cap = round, line width = 3, color = blue ] (0,1.5) --+(0.5, 0) ; 
\draw [ line cap = round, line width = 3, color = blue ] (1,1.5) --+(0, 0.5) ; 
\draw [ line cap = round, line width = 3, color = blue ] (2, 1.5) --+(0, 0.5) ; 
\draw [ line cap = round, line width = 3, color = blue ] (3, 1.5) --+(0, 0.5) ; 
\draw [ line cap = round, line width = 3, color = blue ] (0,2) --+(0.5, 0) ; 
\draw [ line cap = round, line width = 3, color = blue ] (1.5,2) --+(0, 0.5) ; 
\draw [ line cap = round, line width = 3, color = blue ] (2.5,2) --+(0, 0.5) ; 
\draw [ line cap = round, line width = 3, color = blue ] (0, 2.5) --+(0, 0.5) ; 
\draw [ line cap = round, line width = 3, color = blue ] (0.5, 2.5) --+(0.5, 0) ; 
\draw [ line cap = round, line width = 3, color = blue ] (2, 2.5) --+(0, 0.5) ; 
\draw [ line cap = round, line width = 3, color = blue ] (3, 2.5) --+(0.5, 0) ; 
\draw [ line cap = round, line width = 3, color = blue ] (0.5, 3) --+(0, 0.5) ; 
\draw [ line cap = round, line width = 3, color = blue ] (1, 3) --+(0.5, 0) ; 
\draw [ line cap = round, line width = 3, color = blue ] (2.5, 3) --+(0.5, 0) ; 
%
\draw [ line width = 5, color = red ] (0,1) -- +(-0.5, 0.5);
\draw [ line width = 5, color = red ] (2,1) -- +(0.5, 0.5);

\foreach \x in {0, 1, 2, 3}
\foreach \y in {0, 1, 2, 3}
{
\draw (\x, \y) [ fill = black ] circle ( 0.08 cm ) ; 
}

\foreach \x in {0, 1, 2, }
\foreach \y in {0, 1, 2, }
{
\draw (\x+0.5, \y+0.5) [ fill = white ] circle ( 0.15 cm ) ; 
}

\begin{scope} [ xshift = 0.5 cm, yshift = 0.5 cm ]
{
\draw (0, 3) [ fill = white ] circle ( 0.15 cm ) ; 
\draw (0, 3) [ fill = black ] circle ( 0.08 cm ) ; 

\draw (-1, 1) [ fill = white ] circle ( 0.15 cm ) ; 
\draw (-1, 1) [ fill = black ] circle ( 0.08 cm ) ; 

\draw (3, 2) [ fill = white ] circle ( 0.15 cm ) ; 
\draw (3, 2) [ fill = black ] circle ( 0.08 cm ) ; 
}
\end{scope}

\ent
\end{center}
\caption{
A Dimer Covering on  $G^{(2)}$
}
\label{Dimer-Example}
\end{figure}

We shall review 
the known results on the perfect matching problem on 
$G^{(k)}$. 
In 
\cite{NOS}, 
they consider two types of elementary moves 
which transforms a perfect matching on $G^{(k)}$ to a different one, and showed that any two perfect matchings are connected via a sequence of elementary moves. 
Then 
we can construct a Markov chain on 
${\cal M}(G^{(k)})$ 
with the uniform stationary distribution. 
By MCMC simulation, 
we conjecture that all impurities tend to distribute near the terminals,  more precisely,  
the number of perfect matchings 
is maximized when all the impurities are on the terminals.
Ciucu \cite{Ciucu1, Ciucu2}
studied the dimer-monomer problem mainly on the hexagonal lattice and found that monomers interact as if they were the charged particles in 2-dimensional electrostatics. 
Though 
his setting of the problem is different from ours, his works should have something to do with our conjecture. 
In 
\cite{NS1}, 
they studied the one impurity case, and derived a formula for computing the number of perfect matchings if the impurity is arbitrary fixed. 
From that formula, 
the conjecture above is generically correct for 
$k=1$, 
in the sense that 
the probability of finding the impurity on the fixed site 
(under the uniform distribution on 
${\cal M}(G^{(1)})$)
decays exponentially away from the terminal. 

In this paper, 
we study the 
$k \ge 2$
case and derive a formula to compute the number of perfect matchings if the impurities are fixed on the boundary(Theorem
\ref{two impurities}). 
Moreover, 
we study the behavior of the probability of finding the impurity when the size of the graph tends to infinity. 
\\

{\bf Notation}\\
We collect 
the notations frequently used in this paper. \\

(1)
$\triangle_G$
is the Laplacian on a weighted graph
$G$ : 
\begin{equation}
(\triangle_G f)(v)
=
\sum_{w \in V(G)} c_{vw}(f(v) - f(w))
\label{Laplacian}
\end{equation}
where 
$c_{vw}$
is the weight on the edge 
$(v,w) \in E(G)$
which we take 
$c_{vw} = 1$
in this paper. 

$K$
is the Laplacian on 
$G_{1,T}$
restricted on 
$G_1$ : 
\begin{equation}
K :=\triangle_{G_{1,T}} |_{G_1}
=
1_{G_1} 
\triangle_{G_{1,T}} 
1_{G_1}.
\label{K}
\end{equation}
$1_S$
is the characteristic function on a set 
$S$. 

$A(x,y)$ 
or
$A_{x,y}$
is the matrix element of a matrix 
$A$. \\

(2)
$I^{(1)}, I^{(2)}, \cdots, I^{(k)} \in E(G^{(k)})$
are the location of impurities, where 
$I^{(j)} = (I^{(j)}_1, I^{(j)}_2)$, 
$I^{(j)}_i \in V(G_i)$, 
$j=1, 2, \cdots, k$, 
$i=1,2$.
For the one impurity case($k=1$), 
we simply write 
$I = (I_1, I_2) \in E(G^{(1)})$, 
$I_i \in V(G_j)$, 
$i= 1, 2$.

$t_1, t_2, \cdots, t_{2k-1} \in V(G_1)$
are vertices of 
$G_1$
connected to the terminals
$T_1, T_2, \cdots, T_{2k-1}$ 
respectively(Figure \ref{G1}). 
For one impurity case, 
we simply write
$t$(Figure \ref{One-Impurity}). \\
%

\begin{figure}
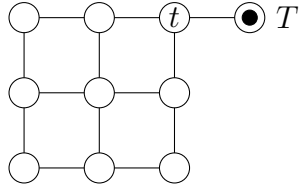

\begin{center}
\bet

\draw (0, 0) grid (2,2) ; 

\draw (2, 2) --+(1, 0); 

\draw (3, 2) [ fill = white ] circle (0.2 cm ) ; 
\draw (3,2) [ fill = black ] circle (0.1 cm) ; 

\foreach \x in {0, 1, 2}
\foreach \y in {0, 1, 2}
{
\draw (\x, \y) [ fill = white ] circle (0.2 cm) ; 
}

\draw (3.5, 2) node {$T$} ; 

\draw (2, 2) node {$t$} ; 
\ent
\end{center}
\caption{
$G_{1, T}$ 
for one impurity case
}
\label{One-Impurity}
\end{figure}

In the following subsections, 
we summarize the results obtained in this paper. 
%
\subsection{One impurity case}
We first recall the result in \cite{NS1} where we set 
$k=1$ 
and counted the number of matchings when the location of the impurity is given. 
\begin{theorem}
\label{one impurity}
Let 
$x \in G_1$. 
Then the number 
$M(x)$ 
of perfect matchings on 
$G^{(1)}$
whose impurity satisfies 
$I_1 = x$
is equal to 
\[
M (x) = 
| (K^{-1})(x, t)
\det K |.
\]
\end{theorem}
If we specify 
$I_1 = x \in G_1$, 
we have four possibilities of putting 
$I_2 \in G_2$,
but the number of perfect matching is independent of the choice of that. 
In Section 3.1, 
we give an alternative proof of 
Theorem \ref{one impurity} 
by using the theory developped by Kenyon-Wilson \cite{KW1, KW2}. 
%
\subsection{Two impurities case}
Set 
$k=2$
and let 
$a, b \in G_1$
be vertices on the boundary 
$\partial G_1$
of 
$G_1$
such that there are no terminals in between
(Figure \ref{G2-Fixed-Impurities}). 
We put 
the impurities such that their location  
$I^{(1)}, I^{(2)}$
satisfies  
$I^{(1)}_1=a$, $I^{(2)}_1=b$
and the other ends 
$I^{(1)}_2$, $I^{(2)}_2$ 
lie on the boundary 
$\partial G_2$ 
of 
$G_2$.
We remark that, 
unlike the one impurity case, 
the number of perfect matchings becomes different when we put 
$I^{(1)}, I^{(2)} \notin \partial G_2$
(Figure \ref{Non-Allowed-Impurities}).
Let 
$\partial C$
be the set of vertices on the boundary 
$\partial G_1$
of 
$G_1$
which lie between 
$a$
and 
$b$
(Figure \ref{Two-Impurities}).
%

\begin{figure}
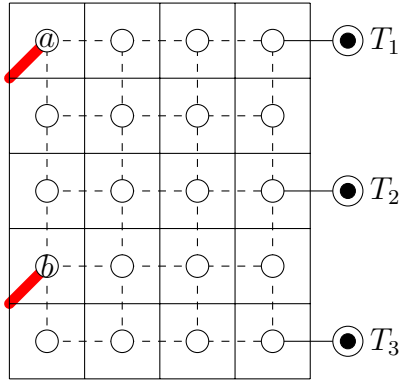

\begin{center}
\bet

\draw [ line cap = round, color = red, line width = 4 ]  (0.5, 4.5) --+(-0.5, -0.5) ; 
\draw [ line cap = round, color = red, line width = 4 ]  (0.5, 1.5) --+(-0.5, -0.5) ; 

\draw (0,0) grid +(4, 5) ; 

\begin{scope} [ xshift = 0.5 cm, yshift = 0.5 cm ] 
{

\draw (0,0) [ dashed ] grid +(3, 4) ; 

\foreach \y in {0, 2, 4}
{
\draw (3, \y) --+(1, 0) ; 
}

\draw (4, 0) [ fill = white ] circle ( 0.2 cm ) ; 
\draw (4, 0) [ fill = black ] circle (0.1 cm ) ; 

\draw (4, 2) [ fill = white ]  circle ( 0.2 cm ) ; 
\draw (4, 2) [ fill = black ] circle (0.1 cm ) ; 

\draw (4, 4) [ fill = white ]  circle ( 0.2 cm ) ; 
\draw (4, 4) [ fill = black ] circle (0.1 cm ) ;

\foreach \x in {0, 1, 2,3}
\foreach \y in {0, 1, 2, 3, 4}
{
\draw (\x, \y) [ fill = white ] circle ( 0.15 cm ) ; 
}

\draw (4.5, 4) node {$T_1$} ; 
\draw (4.5, 2) node {$T_2$} ; 
\draw (4.5, 0) node {$T_3$} ; 

\draw (0, 4) node {$a$} ; 
\draw (0, 1) node {$b$} ;

}
\end{scope}

\ent
\end{center}
\caption{
$G^{(2)}$ 
with fixed impurities
}
\label{G2-Fixed-Impurities}
\end{figure}
%

\begin{figure}
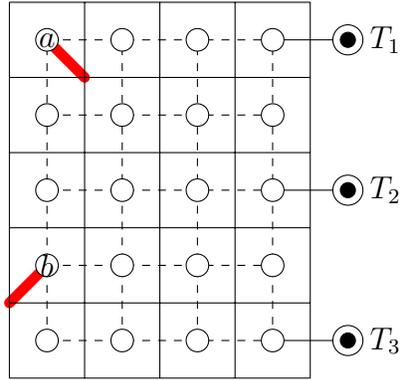

\begin{center}
\bet

\draw [ line cap = round, color = red, line width = 4 ]  (0.5, 4.5) --+(0.5, -0.5) ; 
\draw [ line cap = round, color = red, line width = 4 ]  (0.5, 1.5) --+(-0.5, -0.5) ; 
\draw (0,0) grid +(4, 5) ; 

\begin{scope} [ xshift = 0.5 cm, yshift = 0.5 cm ] 
{

\draw (0,0) [ dashed ] grid +(3, 4) ; 

\foreach \y in {0, 2, 4}
{
\draw (3, \y) --+(1, 0) ; 
}

\draw (4, 0) [ fill = white ] circle ( 0.2 cm ) ; 
\draw (4, 0) [ fill = black ] circle (0.1 cm ) ; 

\draw (4, 2)  [ fill = white ] circle ( 0.2 cm ) ; 
\draw (4, 2) [ fill = black ] circle (0.1 cm ) ; 

\draw (4, 4)  [ fill = white ] circle ( 0.2 cm ) ; 
\draw (4, 4) [ fill = black ] circle (0.1 cm ) ;

\foreach \x in {0, 1, 2, 3}
\foreach \y in {0, 1, 2, 3, 4}
{
\draw (\x, \y) [ fill = white ] circle ( 0.15 cm ) ; 
}

\draw (4.5, 4) node {$T_1$} ; 
\draw (4.5, 2) node {$T_2$} ; 
\draw (4.5, 0) node {$T_3$} ; 

\draw (0, 4) node {$a$} ; 
\draw (0, 1) node {$b$} ;

}
\end{scope}

\ent
\end{center}
\caption{
This impurity configuration is not allowed for Theorem 1.2.
}
\label{Non-Allowed-Impurities}
\end{figure}
%

\begin{figure}
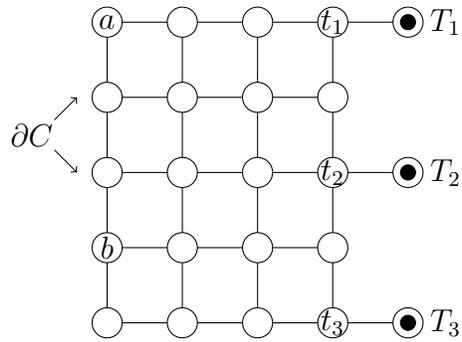

\begin{center}
\bet

\draw (0,0) grid +(3, 4) ; 

\foreach \y in {0, 2, 4}
{
\draw (3, \y) --+(1, 0) ; 
}

\draw (4, 0) [ fill = white ] circle ( 0.2 cm ) ; 
\draw (4, 0) [ fill = black ] circle (0.1 cm ) ; 

\draw (4, 2)  [ fill = white ] circle ( 0.2 cm ) ; 
\draw (4, 2) [ fill = black ] circle (0.1 cm ) ; 

\draw (4, 4)  [ fill = white ] circle ( 0.2 cm ) ; 
\draw (4, 4) [ fill = black ] circle (0.1 cm ) ;

\foreach \x in {0, 1, 2, 3}
\foreach \y in {0, 1, 2, 3, 4}
{
\draw (\x, \y) [ fill = white ] circle ( 0.2 cm ) ; 
}

\draw (4.5, 4) node {$T_1$} ; 
\draw (4.5, 2) node {$T_2$} ; 
\draw (4.5, 0) node {$T_3$} ; 

\draw (3, 4) node {$t_1$} ; 
\draw (3, 2) node {$t_2$} ; 
\draw (3, 0) node {$t_3$} ; 

\draw (0, 4) node {$a$} ; 
\draw (0, 1) node {$b$} ; 

\draw [ <- ] (-0.4, 2) --+(-0.3, 0.3) ; 
\draw [ <- ] (-0.4, 3) --+(-0.3, -0.3) ;  

\draw (-1, 2.5) node {$\partial C$} ;

\ent
\end{center}
\caption{
$G_{1, T}$ 
for two impurities
}
\label{Two-Impurities}
\end{figure}
%


%
\begin{theorem}
\label{two impurities}
Suppose 
the two impurities satisfy 
$I^{(1)}_1 = a$, 
$I^{(2)}_1 = b$
and 
$I^{(j)}_2 \in \partial G_2$, $j=1,2$.
Then the number 
$M(a,b)$
of the corresponding perfect matchings is given by 
\[
M(a,b)
=
A(a,b)
:= 
\left|
\det 
\left(
\begin{array}{ccc}
L_{a, t_1} & L_{a, t_2} & L_{a, t_3} \\
L_{\partial C, t_1} & L_{\partial C, t_2} & L_{\partial C, t_3} \\
L_{b, t_1} & L_{b, t_2} & L_{b, t_3} \\
\end{array}
\right)
\det (K)
\right|
\]
where
\beq
L_{x, t_i} &=& 
(K^{-1})(x, t_i),
\quad
x = a, b, 
\quad
i=1, 2, 3
\\
L_{\partial C, t_i} &=& 
\sum_{y \in \partial C} 
(K^{-1})(y, t_i),
\quad
i=1, 2, 3.
\eeq
\end{theorem}
This determinantal expression implies that impurities are repulsive each other. 
We give two proofs of  
Theorem \ref{two impurities} : 
one is due to the hitting matrix method by Fomin\cite{Fomin} in Section 2, and the other one by Kenyon-Wilson \cite{KW1, KW2} in Section 3.2. 
If 
$I^{(j)}_2 \notin \partial G_2$
(as in Figure \ref{Non-Allowed-Impurities}), 
then we have another formula given in Section 3.4. 
Essentially the same formula also holds for 
$k \ge 3$ 
case to be shown in Section 3.5.  
However, 
if some impurities are not on the boundary 
(i.e., 
$I_1^{(j)} \notin \partial G_1$
and
$I_2^{(j)} \notin \partial G_2$), 
then our methods of proof do not apply and it would be difficult to compute the number of perfect matchings. 
On the other hand, if 
$G_1$
is the subgraph of the one-dimensional chain(Figure \ref{One-Dimensional-Chain}), 
we can compute the number of perfect matchings for arbitrary position of impurities(Section 3.6). 
%

\begin{figure}
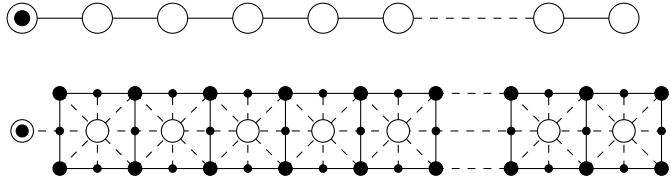

\begin{center}
\bet

\begin{scope} [ xshift = 0.5 cm, yshift = 2 cm ] 
{

\foreach \x in {-1, 0, 1, 2, 3} 
{
\draw (\x, 0) --+(1, 0) ; 
\draw [ fill = white ] (\x, 0) circle (0.2 cm ) ; 
}

\draw [ fill = black ] (-1, 0) circle (0.1 cm) ;  

\draw [ dashed ] (4, 0) --+(2, 0) ; 
\draw (6, 0) -- +(1, 0) ; 
\draw [ fill = white ] (4,0) circle (0.2 cm) ; 
\draw [ fill = white ] (6,0) circle (0.2 cm) ; 
\draw [ fill = white ] (7,0) circle (0.2 cm) ; 

}
\end{scope}

\draw (0, 0) grid +(5, 1) ; 
\draw (6, 0) grid +(2, 1) ; 
\draw [ dashed ] (5, 0) -- +(1, 0) ; 
\draw [ dashed ] (5, 1) -- +(1, 0) ;

\foreach \x in {0, 1, 2, 3, 4, 5, 6, 7}
\foreach \y in {0, 1}
{
\draw [ fill = black ] (\x, \y) circle (0.09cm) ; 
}

\foreach \y in {0, 1}
{
\draw [ fill = black ] (8, \y) circle (0.09cm) ; 
}

\foreach \x in {0, 1, 2, 3, 4, 5, 6, 7, 8}
{
\draw [ fill = black ] (\x, 0.5) circle (0.05cm) ; 
}

\foreach \x in {0, 1, 2, 3, 4, 6, 7}
\foreach \y in {0, 1}
{
\draw [ fill = black ] (\x + 0.5, \y) circle (0.05cm) ; 
}

\foreach \x in {0, 1, 2, 3, 4, 6, 7}
{
\draw [ dashed ] (\x, 0) --++(1, 1) ; 
\draw [ dashed ] (\x, 1) --++(1, -1) ; 
}


\begin{scope} [ xshift = 0.5 cm, yshift = 0.5 cm ] 
{

\draw [ dashed ] (-1, 0) --+(8.5, 0) ; 

\foreach \x in {0, 1, 2, 3, 4, 6, 7}
{
\draw [ dashed ] (\x, -0.5) --+(0, 1) ; 
}

\foreach \x in {-1, 0, 1, 2, 3} 
{
\draw [ fill = white ] (\x, 0) circle (0. 15 cm ) ; 
}
\draw [ fill = black ] (-1, 0) circle (0.08 cm) ;  

\draw [ fill = white ] (4,0) circle (0.15 cm) ; 
\draw [ fill = white ] (6,0) circle (0.15 cm) ; 
\draw [ fill = white ] (7,0) circle (0.15 cm) ;

}
\end{scope}

\ent
\end{center}
\caption{
The case where 
$G_1$
is an one-dimensional chain and the corresponding 
$G^{(1)}$. 
}
\label{One-Dimensional-Chain}
\end{figure}
%

%
\subsection{Large size limit}
In the one impurity case, 
we consider the limit of 
$M(x)$ 
as the size of 
$G_1$
tends to infinity, for the following two graphs.  
\[
G_1=
\cases{ 
G^{sq}_1(n) := \{ 
(x, y) \, | \, 
x = 1,2, \cdots, n, \; 
y = 1, 2, \cdots, n
\}
&
(2-dim) \cr
G^{ch}_1(n)
:= \{1, 2, \cdots, n \}
&
(1-dim) \cr
}
\]
We can also study 
the case where 
$G_1$
is the 
$n \times m$-rectangle,
provided 
$n$, $m$
grow proportionally.
We set 
the problem for the 2-dim case below, but the 1-dim case is formulated similarly. 
Connect the terminal 
$T$
to  
$r=(1,1) \in G_1^{sq}(n)$
(for 1-dim, 
$r=1 \in G_1^{ch}(n)$). 
We consider the following two problems. \\

(1)
Compute 
$\lim_{n \to \infty}{\bf P}\left(
I_1 = (x,y)
\right)$
for fixed
$(x,y) \in G^{sq}_1(n)$.
\\

(2)
We scale
$G_1^{sq}(n)$
by
$\frac 1n$
so that it is contained by the unit square, 
and consider the probability of finding the impurity on a small region in it. 
In other words, 
for any fixed 
$0 \le c_1 < c_2 \le 1$, 
$0 \le d_1 < d_2 \le 1$
we would like to compute 
\beq
\mu([c_1, c_2] \times[d_1, d_2])
:=
\lim_{n \to \infty}
{\bf P}\left(
I_1 \in 
[c_1 n, c_2 n] \times [d_1 n, d_2 n] 
\right)
\eeq
of a measure 
$\mu$
on 
$[0,1]^2$. 
\\
For the first problem, 
\begin{theorem}
\label{problem1}
(1)
(Example 3.6 in \cite{NS1})
For the one-dimensional chain
$G^{ch}_1(n)$, 
for fixed 
$j \in G^{ch}_1(n)$, 
we have
\[
{\bf P}\left(
I_1 =  j
\right)
=
\frac 14
\lambda_+^{-j}
(1 + o(1)), 
\quad
n \to \infty
\]
where
$\lambda_+ := 2 + \sqrt{3}$. 
\\
(2)
For the two-dimensional grid
$G^{sq}_1(n)$, 
for any fixed 
$(x,y) \in G^{sq}_1(n)$, 
we have
\[
\lim_{n \to \infty}
{\bf P}\left(
I_1 = (x,y)
\right)
=0.
\]
\end{theorem}
Theorem \ref{problem1}
has the following implications. 
(i)
for the one-dimensional chain, 
the probability of finding the impurity decays exponentially and it is localized near the terminal, 
(ii)
for the two-dimensional grid, 
the probability spread if the size of the grid is large. \\
The perfect matching on 
$G^{(1)}$
is determined by a spanning tree
$T$ 
on a graph
$G_{1, R}$
defined in Section 2.1(Theorem \ref{bijection}).
And the difference between the one- and two-dimensional cases comes from that of the expectation value of the length 
$l_T$ 
of 
$T$
(Proposition \ref{G}). 
In one-dimensional case, 
it is bounded with respect to 
$n$, 
while it diverges in the logarithmic order in the two-dimensional case(Proposition \ref{G}(2)). 
This observation 
together with the Chebyshev's inequality also solves the second problem : 
\begin{theorem}
\label{problem2}
In both cases, 
$\mu$
is equal to the delta measure on the origin. 
\end{theorem}
In \cite{CK}, 
they studied the lozenge tiling with a gap and showed that the correlation function behaves like 
$\frac 1r$, 
where 
$r$
is the distance between the gap and the boundary. 
In our case, 
the impurity tends to be attracted to the terminal 
$T$
and not to the whole boundary, 
so that the situation is different. 
In the following sections, 
we prove those theorems mentioned above. 
%

%
\section{Hitting Matrix}
\subsection{An extension of Temperley Bijection}
We first recall the results in 
\cite{NS2}.
We consider 
an imaginary vertex $R$
(called the root)
in the outer face of 
$G_{1, T}$
and let 
$G_{1, T, R}$
be the graph obtained by connecting all vertices in  
$\partial G_{1, T}$
to the root 
$R$(Figure \ref{G1TR}). 
%

\begin{figure}
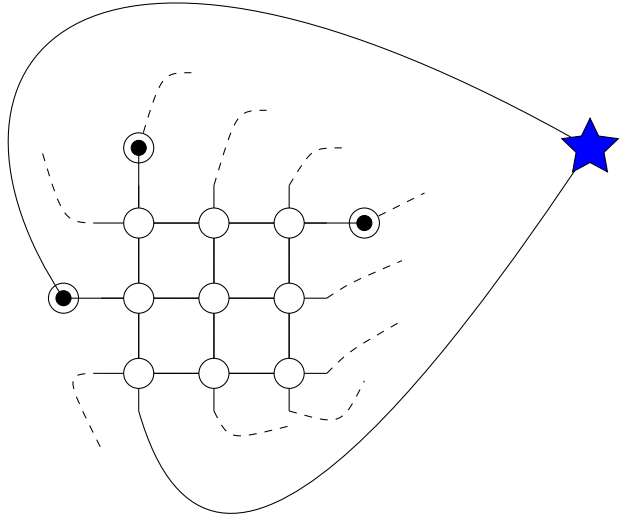

\begin{center}
\bet

\draw (0,0) grid +(2, 2) ;

\draw (0, 3) circle ( 0.2 cm ) ; 
\draw (0, 3) [ fill = black ] circle ( 0.1 cm ) ; 

\draw (-1, 1) circle ( 0.2 cm ) ; 
\draw (-1, 1) [ fill = black ] circle ( 0.1 cm ) ; 

\draw (3, 2) circle ( 0.2 cm ) ; 
\draw (3, 2) [ fill = black ] circle ( 0.1 cm ) ; 

\draw (0, 2) --+(0, 1) ; 
\draw (0, 1) --+(-1, 0) ; 
\draw (2, 2) --+(1, 0) ;

\foreach \x in {0, 1, 2}
\foreach \y in {0, 1, 2}
{
\draw (\x, \y) --+(-0.5, 0) ; 
\draw (\x, \y) --+(0.5, 0) ; 
\draw (\x, \y) --+(0, -0.5) ; 
\draw (\x, \y) --+(0, 0.5) ; 
}

\foreach \x in {0, 1, 2}
\foreach \y in {0, 1, 2}
{
\draw (\x, \y)[ fill = white ] circle (0.2 cm ) ; 
}

\draw (-1, 1) .. controls  (-3, 4) and (-1, 7)  .. (6, 3) ; 
\draw (0, -0.5) .. controls (1, -4) and (4, 0)  .. (6, 3) ; 

\draw [ dashed ] (-0.5, 0) .. controls (-1, 0) .. (-0.5, -1) ; 
\draw [ dashed ] (-0.5, 2) .. controls (-1, 2) .. (-1.3, 3) ; 
\draw [ dashed ] (0, 3) .. controls (0.3, 4) .. (0.7, 4) ; 
\draw [ dashed ] (1, 2.5) .. controls (1.3, 3.5) .. (1.7, 3.5) ; 
\draw [ dashed ] (2, 2.5) .. controls (2.3, 3.) .. (2.7, 3.) ; 
\draw [ dashed ] (3, 2) .. controls (3.2, 2.1) .. (3.8, 2.4) ; 
\draw [ dashed ] (2.5, 1) .. controls (2.8, 1.2) .. (3.5, 1.5) ; 
\draw [ dashed ] (2.5, 0) .. controls (2.8, 0.3) .. (3.5, 0.7) ; 

\draw [ dashed ] (2, -0.5) .. controls (2.7, -0.7) .. (3, -0.1) ; 
\draw [ dashed ] (1, -0.5) .. controls (1.2, -0.9) .. (2, -0.7) ; 

\node [ star, fill = blue, star point height = .2 cm, minimum size = 0.5 cm, draw  ] at (6, 3) {} ; 

\ent
\end{center}
\caption{
$G_{1, T, R}$ : 
all vertices on $\partial G_{1, T}$ are connected to $R$.
}
\label{G1TR}
\end{figure}
%


%
{\bf TI-tree}
is a tree on 
$G_{1, T,R}$
starting from the root 
$R$, 
which directly connects it to a terminal, and ends at a vertex of 
$G_{1, T}$. 
{\bf TO-tree}
is a tree on 
$G_{1, T, R}$
starting at a terminal and ends at the root 
$R$
through a boundary vertex. 
{\bf IO-tree}
is a tree on 
$G_{1, T, R}$
starting at a vertex of 
$G_1$ and ends at the root 
$R$ 
through a boundary vertex(Figure \ref{G1TR-Tree}). 
Theorem 2.8 in \cite{NS2}, 
in a slightly different form, is : 
%

\begin{figure}
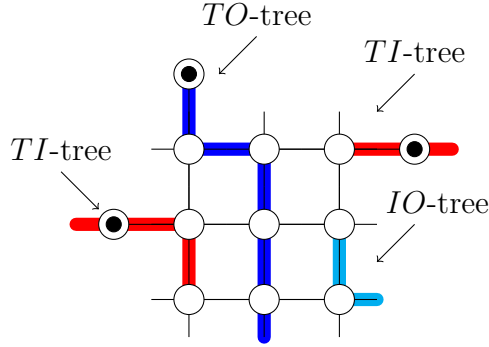

\begin{center}
\bet

\draw (0,0) grid +(2, 2) ; 

\draw [ line cap = round, color = red, line width = 5  ] (0, 0) --++(0, 1) --++(-1.5, 0); 
\draw [ <- ] (-1.2, 1.2) --+(-0.5, 0.5) ;  
\draw (-1.7, 2) node {$TI$-tree} ;

\draw [ line cap = round, color = red, line width = 5  ] (2,2) --++(1.5, 0); 
\draw [ <- ] (2.5, 2.5) --+(0.5, 0.5) ; 
\draw (3, 3.3) node {$TI$-tree} ; 

\draw [ line cap = round, color = blue, line width = 5  ] (0,3) --++(0,-1) --++(1,0) --++(0, -2.5)  ; 
\draw [ <- ] (0.4, 3) --+(0.5, 0.5) ; 
\draw (0.9, 3.8) node {$TO$-tree} ;

\draw [ line cap = round, color = cyan, line width = 5  ] (2, 1) --++(0,-1) --++(0.5,0); 
\draw [ <- ] (2.5, 0.5) --+(0.5, 0.5) ; 
\draw (3.3, 1.3) node {$IO$-tree} ; 

\draw (0, 2) --+(0, 1) ; 
\draw (0, 1) --+(-1, 0) ; 
\draw (2, 2) --+(1, 0) ; 

\draw (0, 3) [ fill = white ] circle ( 0.2 cm ) ; 
\draw (0, 3) [ fill = black ] circle ( 0.1 cm ) ; 

\draw (-1, 1) [ fill = white ]  circle ( 0.2 cm ) ; 
\draw (-1, 1) [ fill = black ] circle ( 0.1 cm ) ; 

\draw (3, 2)  [ fill = white ] circle ( 0.2 cm ) ; 
\draw (3, 2) [ fill = black ] circle ( 0.1 cm ) ;

\foreach \x in {0, 1, 2}
\foreach \y in {0, 1, 2}
{
\draw (\x, \y) --+(-0.5, 0) ; 
\draw (\x, \y) --+(0.5, 0) ; 
\draw (\x, \y) --+(0, -0.5) ; 
\draw (\x, \y) --+(0, 0.5) ; 
}

\foreach \x in {0, 1, 2}
\foreach \y in {0, 1, 2}
{
\draw (\x, \y)[ fill = white ] circle (0.2 cm ) ; 
}

\ent
\end{center}
\caption{
Example of 
TI-tree, TO-tree, and IO-tree. 
$R$
is omitted. 
}
\label{G1TR-Tree}
\end{figure}
%

%
\begin{theorem}
\label{bijection}
We have a bijection between the following two sets. 
\beq
{\cal M}(G^{(k)})
&:=&
\{ \mbox{ perfect matchings on }G^{(k)} \}
\\
{\cal F}(G^{(k)}, Q)
&:=&
\{
(T, S, \{ e_j \}_{j=1}^k)
\, | \,
\mbox{ $T$ : spanning tree on 
$G_{1, T, R}$, }
\\
&& \qquad
\mbox{
$S$ : spanning forest on 
$G_2$, }
\\
&& \qquad
\mbox{ 
$\{ e_j \}_{j=1}^k$ : 
configuration of impurites, 
with condition ${\bf (Q)}$  }
\}
\eeq
${\bf Q}$ : 
\\
(1)
$T$
is composed of 
$k$ TI-trees, 
$(k-1)$ TO-trees, 
and the other ones are 
$IO$-trees, 
\\
(2)
$S$
is composed of $k$ trees, 
\\
(3)
$T$, $S$ 
are disjoint of each other, and the 
$k$ TI-trees of 
$T$
and the $k$ trees of 
$S$
are paired by impurities.
\end{theorem}
Figure \ref{G2-Forest}
is the spanning forest 
$S$ 
on 
$G_2$ 
corresponding to the spanning tree of 
$G_{1, T, R}$ 
shown in Figure \ref{G1TR-Tree}, 
and Figure \ref{Matching} shows the corresponding perfect matching of $G^{(2)}$. 
Sometimes 
TI-tree and TO-tree sticks together to form another tree(Figure \ref{TITO}). 
For one impurity case($k=1$), 
$T$
is composed of a TI-tree and some IO-trees and the bijection in 
Theorem \ref{bijection} is the same as  Temperley's bijection(Figure \ref{One-Impurity}). 
%

\begin{figure}
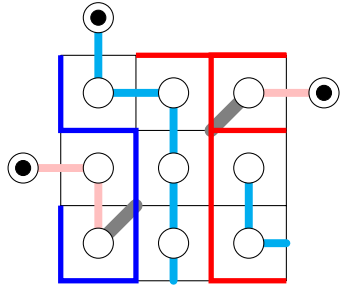

\begin{center}
\bet

\draw (0,0) grid +(3, 3) ; 

\begin{scope} [ xshift = 0.5 cm, yshift = 0.5 cm ] 
{

\draw (0, 2) --+(0, 1) ; 
\draw (0, 1) --+(-1, 0) ; 
\draw (2, 2) --+(1, 0) ; 

\draw [ line cap = round, color = pink, line width = 3  ] (0,0) --++(0,1) --++(-1,0); 

\draw [ line cap = round, color = pink, line width = 3  ] (2,2) --++(1, 0); 

\draw [ line cap = round, color = cyan, line width = 3  ] (0,3) --++(0,-1) --++(1,0) --++(0, -2.5)  ; 

\draw [ line cap = round, color = cyan, line width = 3  ] (2, 1) --++(0,-1) --++(0.5,0); 

\draw [ line cap = round, color = gray, line width = 5 ] (0, 0) --+(0.5, 0.5) ; 
\draw [ line cap = round, color = gray, line width = 5 ] (2, 2) --+(-0.5, -0.5) ; 

\foreach \x in {0, 1, 2}
\foreach \y in {0, 1, 2}
{
\draw (\x, \y) [ fill = white ] circle (0.2 cm) ; 
}

\draw (0, 3)  [ fill = white ] circle ( 0.2 cm ) ; 
\draw (0, 3) [ fill = black ] circle ( 0.1 cm ) ; 

\draw (-1, 1)  [ fill = white ] circle ( 0.2 cm ) ; 
\draw (-1, 1) [ fill = black ] circle ( 0.1 cm ) ; 

\draw (3, 2)  [ fill = white ] circle ( 0. 2 cm ) ; 
\draw (3, 2) [ fill = black ] circle ( 0.1 cm ) ;

}
\end{scope}

\draw [ color = blue, line width = 2 ]  (0, 1) --++(0, -1) --++(1, 0) --++(0, 2) --++(-1, 0) --++(0, 1) ; 
\draw [ color = red, line width = 2 ] (1, 3) --++(2, 0) --++(-1, 0) --++(0, -3) --++(1, 0) ; 
\draw [ color = red, line width = 2 ] (2, 2) --++(1, 0) ;

\ent
\end{center}
\caption{
The spanning forest on 
$G_2$ 
corresponding to the spanning tree in Figure 2.2. 
}
\label{G2-Forest}
\end{figure}
%

\begin{figure}
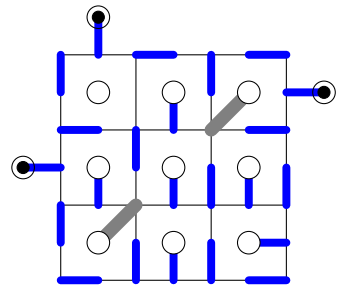

\begin{center}
\bet

\draw (0,0) grid +(3, 3) ;

\draw [ line cap = round, color = gray, line width = 5 ] (0.5, 0.5) --+(0.5, 0.5) ; 
\draw [ line cap = round, color = gray, line width = 5 ] (2, 2) --+(0.5, 0.5) ; 

\draw [ line cap = round, color = blue, line width = 3 ] (0,0) --+(0.5,0) ; 
\draw [ line cap = round, color = blue, line width = 3 ] (1,0) --+(0,0.5) ;
\draw [ line cap = round, color = blue, line width = 3 ] (2,0) --+(0,0.5) ;
\draw [ line cap = round, color = blue, line width = 3 ] (3,0) --+(-0.5,0) ;

\draw [ line cap = round, color = blue, line width = 3 ] (0, 1) --+(0,-0.5) ;
\draw [ line cap = round, color = blue, line width = 3 ] (2, 1) --+(0,0.5) ;
\draw [ line cap = round, color = blue, line width = 3 ] (3, 1) --+(0,0.5) ;

\draw [ line cap = round, color = blue, line width = 3 ] (0, 2) --+(0.5,0) ;
\draw [ line cap = round, color = blue, line width = 3 ] (1, 2) --+(0,-0.5) ;
\draw [ line cap = round, color = blue, line width = 3 ] (3, 2) --+(-0.5,0) ;

\draw [ line cap = round, color = blue, line width = 3 ] (0, 3) --+(0,-0.5) ;
\draw [ line cap = round, color = blue, line width = 3 ] (1, 3) --+(0.5,0) ;
\draw [ line cap = round, color = blue, line width = 3 ] (2, 3) --+(0,-0.5) ;
\draw [ line cap = round, color = blue, line width = 3 ] (3, 3) --+(-0.5,0) ;

\draw [ line cap = round, color = blue, line width = 3] (0, 1.5) --+(-0.5,0) ; 
\draw [ line cap = round, color = blue, line width = 3 ] (0.5, 3) --+(0,0.5) ;
\draw [ line cap = round, color = blue, line width = 3 ] (3, 2.5) --+(0.5,0) ;

\draw [ line cap = round, color = blue, line width = 3 ] (1.5, 0) --+(0,0.5) ;
\draw [ line cap = round, color = blue, line width = 3 ] (3, 0.5) --+(-0.5,0) ;

\draw [ line cap = round, color = blue, line width = 3 ] (0.5, 1) --+(0,0.5) ;
\draw [ line cap = round, color = blue, line width = 3 ] (1.5, 1) --+(0,0.5) ;
\draw [ line cap = round, color = blue, line width = 3 ] (2.5, 1) --+(0,0.5) ;

\draw [ line cap = round, color = blue, line width = 3 ] (1.5, 2) --+(0,0.5) ;

\begin{scope} [ xshift = 0.5 cm, yshift = 0.5 cm ] 
{
\foreach \x in {0, 1, 2}
\foreach \y in {0, 1, 2}
{
\draw (\x, \y) [ fill = white ] circle (0.15 cm) ; 
}

\draw (0, 3) circle ( 0.15 cm ) ; 
\draw (0, 3) [ fill = black ] circle ( 0.08 cm ) ; 

\draw (-1, 1) circle ( 0.15 cm ) ; 
\draw (-1, 1) [ fill = black ] circle ( 0.08 cm ) ; 

\draw (3, 2) circle ( 0.15 cm ) ; 
\draw (3, 2) [ fill = black ] circle ( 0.08 cm ) ; 


}
\end{scope}

\ent
\end{center}
\caption{
The corresponding perfect matching on $G^{(2)}$
}
\label{Matching}
\end{figure}
%

\begin{figure}
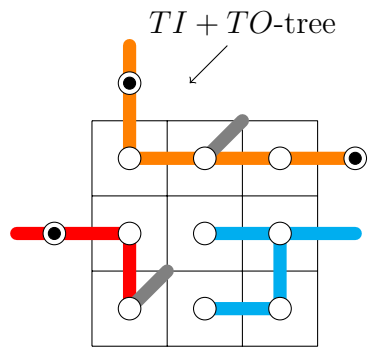

\begin{center}
\bet

\draw (0,0) grid +(3, 3) ; 

\begin{scope} [ xshift = 0.5 cm, yshift = 0.5 cm ] 
{

\draw (0, 2) --+(0, 1) ; 
\draw (0, 1) --+(-1, 0) ; 
\draw (2, 2) --+(1, 0) ; 

\draw [ line cap = round, color = red, line width = 5 ] (0,0) --++(0, 1) --++(-1.5, 0) ; 
\draw [ line cap = round, color = cyan, line width = 5 ] (1, 0) --++(1, 0) --++(0, 1) ; 
\draw [ line cap = round, color = cyan, line width = 5 ] (1,1) --++(2, 0) ;
\draw [ line cap = round, color = orange, line width = 5 ] (0, 3.5) --++(0, -1.5) --++(3, 0) ; 
\draw [ <- ] (0.8, 3) --+(0.5, 0.5) ; 
\draw (1.5, 3.8) node {$TI+TO$-tree} ; 

\draw [ line cap = round, color = gray, line width = 5 ] (0, 0) --+(0.5, 0.5) ; 
\draw [ line cap = round, color = gray, line width = 5 ] (1, 2) --+(0.5, 0.5) ; 

\foreach \x in {0, 1, 2}
\foreach \y in {0, 1, 2}
{
\draw (\x, \y) [ fill = white ] circle (0.15 cm) ; 
}

\draw (0, 3) [ fill = white ]  circle ( 0.15 cm ) ; 
\draw (0, 3) [ fill = black ] circle ( 0.08 cm ) ; 

\draw (-1, 1)  [ fill = white ] circle ( 0.15 cm ) ; 
\draw (-1, 1) [ fill = black ] circle ( 0.08 cm ) ; 

\draw (3, 2)  [ fill = white ] circle ( 0.15 cm ) ; 
\draw (3, 2) [ fill = black ] circle ( 0.08 cm ) ; 

}
\end{scope}

\ent
\end{center}
\caption{
$TI + TO$-tree
}
\label{TITO}
\end{figure}
%

\begin{figure}
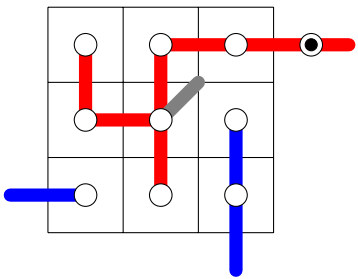

\begin{center}
\bet

\draw (0,0) grid +(3, 3) ; 

\begin{scope} [ xshift = 0.5 cm, yshift = 0.5 cm ] 
{

\draw (2, 2) --+(1, 0) ; 

\draw [ line cap = round, color = blue, line width = 5 ] (0,0) --++(-1, 0) ; 
\draw [ line cap = round, color = blue, line width = 5 ] (2, 1) --++(0, -2) ; 
\draw [ line cap = round, color = red, line width = 5 ] (3.5, 2) --++(-2.5, 0) --++(0, -2) ; 
\draw [ line cap = round, color = red, line width = 5 ] (0, 2) --++(0, -1) --++(1, 0) ; 

\draw [ line cap = round, color = gray, line width = 5 ] (1, 1) --+(0.5, 0.5) ;

\foreach \x in {0, 1, 2}
\foreach \y in {0, 1, 2}
{
\draw (\x, \y) [ fill = white ] circle (0.15 cm) ; 
}

\draw (3, 2) [ fill = white ] circle ( 0.15 cm ) ; 
\draw (3, 2) [ fill = black ] circle ( 0.08 cm ) ;

}
\end{scope}

\ent
\end{center}
\caption{
The case of one impurity
}
\label{One-Impurity2}
\end{figure}
%


%
\subsection{Proof of Theorem \ref{two impurities}
using the hitting matrix method}
By Theorem \ref{bijection}, 
$M(a,b)$
is equal to the number of elements in 
${\cal F}(G^{(2)}, Q)$
such that the corresponding spanning tree of 
$G_{1, T, R}$ 
is composed of 
(i)
a TI-tree connecting 
$a$ and $T_1$, 
(ii)
a TI-tree connecting 
$b$ and $T_3$, and 
(iii)
a TO-tree connecting an element of 
$\partial C$
and
$T_2$ : 
\begin{eqnarray}
M(a,b)
&=&
\sharp \{
F \, | \, 
F 
\mbox{ is a spanning tree on $G_{1,T, R}$ 
composed of }
\nonumber
\\
&&
\mbox{
(i) TI-tree connecting $a$ and $T_1$, 
(ii) TI-tree connecting $b$ and $T_3$}
\nonumber
\\
&&
\mbox{ 
(iii) TO-tree connecting $\partial C$ and $T_2$, 
and
}
\nonumber
\\
&&
\mbox{
(iv) other IO-trees }
\}
\label{forest}
\end{eqnarray}
Let 
$G_{1, R}$
be the dual graph of 
$G_2$
in the sense that we regard the outer face as a face of 
$G_2$
and put a dual vertex 
$R$
there, 
which is called the root(Figure \ref{G1-Bar}).
$G_{1, R}$
is also obtained  
by identifying all terminals in 
$G_{1, T, R}$ 
with 
$R$.  
%

\begin{figure}
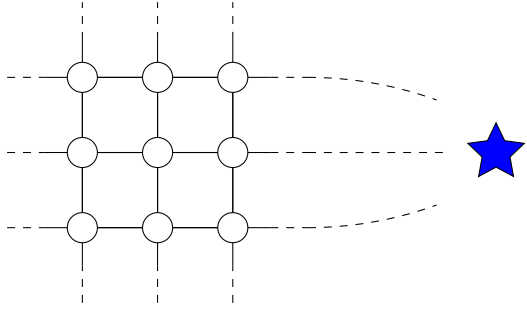

\begin{center}
\bet

\draw (0,0) grid +(2, 2) ; 

\foreach \x in {0, 1, 2}
{
\draw (-0.5, \x) --+(3, 0) ; 
\draw (\x, -0.5) --+(0, 3) ; 

\draw [ dashed ] (-1, \x) --+(0.5, 0) ; 
\draw [ dashed ] (2.5, \x) --+(0.5, 0) ; 
\draw [ dashed ] (\x, -1) --+(0, 0.5) ; 
\draw [ dashed ] (\x, 2.5) --+(0, 0.5) ; 
}

\foreach \x in {0, 1, 2}
\foreach \y in {0, 1, 2}
{
\draw [ fill = white ]  (\x, \y) circle (0.2cm) ; 
}

\node [ star, fill = blue, star point height = .2 cm, minimum size = 0.5 cm, draw  ] at (5.5, 1) {} ; 

\draw [ dashed ] (3, 2) arc (90 : 70 : 5) ; 
\draw [ dashed ] (3, 1)  --+ (1.8, 0) ; 
\draw [ dashed ] (3, 0) arc (270 : 290 : 5) ; 

\ent
\end{center}
\caption{
In 
$G_{1, R}$, 
all vertices on the boundary of 
$G_1$
are connected to $R$ such that their degree $=4$.
}
\label{G1-Bar}
\end{figure}
%


%
Then the RHS of 
(\ref{forest})
is equal to the number of spanning trees
$T$ 
of 
$G_{1, R}$
with the following property : 
cutting all the edges connecting to 
$R$, 
yields a spanning forest
$T_1$ 
of 
$G_1$ 
which satisfies 
(i) $a$ and $t_1$
are connected, 
(ii)
$b$ and $t_3$ 
are connected, 
(iii)
$\partial C$ and $t_2$ 
are connected : 
\beq
M(a,b) &=&
\{ T \, | \, 
T_1 \mbox{ is a spanning forest on $G_1$ such that }
\\
&& 
\mbox{ 
(i) $a$ is connected to $t_1$, 
(ii) $b$ is connected to $t_3$, }
\\
&&
\mbox{
(iii) $\partial C$ is connected to $t_2$, }
\}
\eeq
By taking acount of 
Wilson's algorithm \cite{Wilson}
generating the spanning trees on 
$G_{1, R}$
uniformly at random, we have 
\beq
&&
\frac {M(a,b)}
{\sharp \{ \mbox{ spanning trees on } G_{1, R} \} }
\\
&=&
\sum_{\sigma}
(\mbox{ sgn } \sigma)
\;
w(\pi(a, t_{\sigma_1}))
w(\pi(b, t_{\sigma_2}))
w(\pi(\partial C, t_{\sigma_3}))
1_{A(\sigma)}. 
\eeq
In the notation above, 
$\pi(x,y)$
is the walk from $x$ to $y$, 
$| \pi |$
is the length of the walk 
$\pi$, 
$w(\pi) :=4^{- | \pi |}$
is the weight of the walk
$\pi$,
and 
$1_{A(\sigma)}$
is the indicator function of the conditition 
$A(\sigma)$ : 
``
$\pi(b, t_{\sigma_2})$
is disjoint from the loop erased part of 
$\pi(a, t_{\sigma_1})$
and 
$\pi(c, t_{\sigma_3})$
is disjoint from the loop erased part of 
$\pi(b, t_{\sigma_2})$" : 
\beq
A(\sigma) := \{
\pi(b, t_{\sigma_2}) \cap LE(\pi(a, t_{\sigma_1})) = \emptyset, 
\;
\pi(c, t_{\sigma_3}) \cap LE(\pi(b, t_{\sigma_2})) = \emptyset
\}
\eeq
Therefore by 
Theorem 7.2 in \cite{Fomin} 
and the matrix tree theorem : 
$\sharp \{ \mbox{ spanning trees on } G_{1,R} \} = \det (K)$
we have 
\begin{equation}
M(a,b) = 
\left|
\det
\left( \begin{array}{ccc}
H_{a,t_1} & H_{a,t_2} & H_{a,t_3} \\
H_{b,t_1} & H_{b,t_2} & H_{b,t_3} \\
H_{\partial C,t_1} & H_{\partial C,t_2} & H_{\partial C,t_3} 
\end{array} \right)
\det (K)
\right|
\label{hitting-matrix}
\end{equation}
where 
$H_{x, t_i}$
is the hitting probability from 
$x$ 
to 
$t_1$ :
probability of a simple random walk on 
$G_{1, R}$
starting at 
$x$
hitting the root 
$R$
through 
$t_i$, 
and 
\[
H_{\partial C, t_i}
:=
\sum_{z \in \partial C} H_{z, t_i}.
\]
To compute 
$H_{x, y}$, 
let 
$Q$
be a matrix on 
$G_1$ 
given by 
\[
Q_{x, y} :=\cases{
\frac 14 & $(x \sim y)$ \cr
0 & (otherwise) \cr
}
\]
$Q$
is the transition probability matrix of the 
SRW on 
$G_{1, R}$
restricted to 
$G_1$. 
Since 
$4(I - Q) = K$, 
we have 
\beq
H_{x, t_i}
&=&
( (I - Q)^{-1} )_{x,  t_i} \cdot \frac 14
\\
&=&
K^{-1}_{x,t_i}
=
L_{x, t_i}
\eeq
for
$x \in G_1$, $i=1,2,3$.
Substituting it into 
(\ref{hitting-matrix})
yields 
Theorem \ref{two impurities}.
%
\section{Alternative proofs by a Kenyon-Wilson's result}
\subsection{Kenyon-Wilson's results}
We first recall 
the results by 
Kenyon-Wilson\cite{KW1, KW2}.
Let 
$G_C$
be a planar, finite, and weighted graph, and for 
$e=\{ v, w \} \in E(G_C)$, 
let 
$w(e) = c_{vw}$
be the weight on that. 
We choose 
$n$
vertices on the boundary of 
$G_C$, 
number those points anticlockwise, and let 
$N := \{1, 2, \cdots, n \}$
which is called the {\bf node}. 
A graph with such a property is called a {\bf circular planar graph}. 
Let 
$I := V(G_C) \setminus N$. 
We say a spanning forest 
$g$ 
on 
$G_C$
is a {\bf grove} if each tree composing 
$g$
contains at least an element of 
$N$. 
A grove 
$g$
determines a planar partition (non-crossing partition) 
$\sigma$ 
of 
$N$.
We set 
\beq
{\cal P} 
&:=& \{
\mbox{ partitions of  }
\{1, 2, \cdots, n \}
\}
\\
{\cal P}_{pl} 
&:=& \{
\mbox{ planar-partitions of }
\{1, 2, \cdots, n \}
\}
\eeq
For a grove
$g$, 
let 
$w(g) := \prod_{e \in g}w(e)$
be its weight. 
For given planar partition 
$\sigma \in {\cal P}_{pl}$
of 
$N$, 
we consider all groves which realize the partition 
$\sigma$ 
and let 
$Z_{G_C} (\sigma) := \sum_{g \in \sigma} w(g)$
be their weighted sum. 
According to the decomposition 
$V(G_C) = N \cup I$, 
the Laplacian 
$\triangle_{G_C}$
on 
$G_C$
can be written as 
\begin{equation}
\triangle_{G_C} = 
\left( \begin{array}{cc}
F & G \\
H & K 
\end{array} \right).
\label{decomposition}
\end{equation}
We define the response matrix 
$L$
by 
\[
L = G K^{-1} H - F.
\]
$L$
has the following physical meaning.
We regard
$G_C$
as an electrical circuit with 
$c_{vw}$
being the conductance on the edge 
$\{ v, w \}$.  
For 
$i, j \in N$, 
set the voltage 
$f$
on $N$ by 
\[
f(x) = \cases{
1 & $(x = i)$ \cr
0 & (otherwise) \cr
}, 
\quad
x \in N.
\]
Then 
$L_{i, j}$
is equal to the current flowing out of 
$j$. 
$\sigma \in {\cal P}_{pl}$
is said to be {\bf bipartite} if it satisfies the following condition : 
(i) it is possible 
to color the vertices in 
$N$
contiguously into red and blue such that only the nodes of different colors are connected,  
(ii) if a node lies 
between red region and blue region, it can be split into two colors, 
and 
(iii) the number of each part in 
$\sigma$ 
is at most two. 

For bipartite 
$\sigma$, 
let 
$L_{X, Y}$
be the submatrix of 
$L$ 
such that the color of the row(resp. column) is 
$X$
(resp. $Y$).
We then have the following theorem. 
\begin{theorem}
{\bf (Theorem 3.1 in \cite{KW2})}\\
\label{bipartite}
For bipartite 
$\sigma \in {\cal P}_{pl}$, 
we have 
\[
Z(\sigma) = 
| \det L_{R,B} \cdot \det (K) |. 
\]
\end{theorem}
\begin{remark}
Theorem 3.1 in \cite{KW2}
treats more general case called the tripartite partition and write 
$Z(\sigma)$ 
in terms of Pfaffian. 
Theorem \ref{bipartite}
also holds for non-planar graph 
$G$
provided 
$G$
satisfies the following condition : 

{\bf (A)}
For any 
non-planar partition
$\tau \in {\cal P} \setminus {\cal P}_{pl}$, 
we always have
$Z_G(\tau) = 0$. 
\end{remark}
%
\subsection{One impurity case：
alternative proof of Theorem \ref{one impurity}}
For 
$x \in G_1$, 
we define a matrix 
$K_x$
by 
\[
K_x(y,z) := \cases{
K(x,x) + 1 & $((y,z) = (x,x))$ \cr
K(y,z) & (otherwise) \cr
}
\]
\begin{lemma}
\label{one impurity1}
\[
M (x) = 
| (K_x^{-1})(x,t)  \det (K_x) |.
\]
\end{lemma}
\begin{proof}
We construct a circular planar graph
$G_C$
by the following procedure(Figure \ref{GC}) : 
(i) 
add a new vertex 
$X$
to $x$, 
(ii)
add all boundary vertices to 
$G_1$, 
and identify them. 
Let 
$Y$
be the corresponding vertex and we set 
\beq
V(G_C) := I \cup N, 
\quad
I =V(G_1), 
\quad
N := \{ T, X, Y \}.
\eeq
%

\begin{figure}
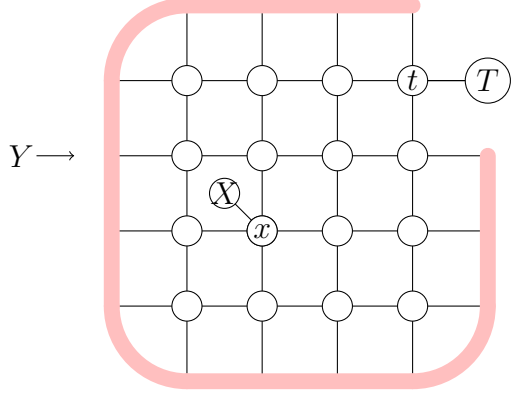

\begin{center}
\bet

\draw (0,0) grid +(3, 3) ; 

\foreach \x in {0, 1, 2, 3}
{
\draw (\x, 0) --+(0, -1) ; 
\draw (\x, 3) --+(0, 1) ; 
\draw (0, \x) --+(-1, 0) ; 
\draw (3, \x) --+(1, 0) ; 
}

\draw (3, 3) --+(1, 0) ; 

\foreach \x in {0, 1, 2, 3}
\foreach \y in {0, 1, 2, 3}
{
\draw  [ fill = white ] (\x, \y) circle (0.2 cm) ; 
}

\draw [ fill = white ] (4, 3) circle (0.3 cm) ; 
\draw (4,3) node {$T$} ; 
\draw (3, 3) node {$t$} ; 

\draw (1, 1) --+ (-0.5, 0.5) ; 

\draw [ fill = white ] (0.5, 1.5) circle (0.2 cm) ; 
\draw (0.5, 1.5) node {$X$} ; 
\draw [ fill = white ] (1,1) circle (0.2 cm) ; 
\draw (1, 1) node {$x$} ; 

\draw [ line cap = round, color = pink, line width = 6 ] (3, 4) --++(-3, 0) ; 
\draw [ line cap = round, color = pink, line width = 6 ] (0, 4) arc (90 : 180 : 1) ; 
\draw [ line cap = round, color = pink, line width = 6 ] (-1, 3) --++(0, -3) ; 
\draw [ line cap = round, color = pink, line width = 6 ] (-1, 0) arc (180 : 270 : 1) ; 
\draw [ line cap = round, color = pink, line width = 6 ] (0, -1) --+(3, 0) ; 
\draw [ line cap = round, color = pink, line width = 6 ] (3, -1) arc (270 : 360 : 1) ; 
\draw [ line cap = round, color = pink, line width = 6 ] (4, 0) --++(0, 2) ; 

\draw [ <- ] (-1.5, 2) --+(-0.5, 0) ; 
\draw (-2.2, 2) node {$Y$} ;

\ent
\end{center}
\caption{
$G_C$ in one impurity case : 
add $X$ to $x$, 
identify all boundary vertices with $Y$
}
\label{GC}
\end{figure}
%

Although 
$G_C$
is not planar, the condition 
(A)
is satisfied, for 
$N$
contains only three elements so that 
we do not have non-planar partitions.
Let 
$\sigma = XT | Y$
be a partition such that 
$X$ and $Y$ are connected and 
$Y$
is isolated. 
By
Theorem \ref{bijection}, 
we have 
$M(x)= Z_{G_C}(\sigma)$.
Writing 
$\triangle_{G_C}$
as in 
(\ref{decomposition}), 
the response matrix
$L$
satisfies 
$L = G K_x^{-1} H - F$.
Note that 
we added an extra vertex
$X$
to 
$x$, 
so that we have to use 
$K_x$
in 
(\ref{decomposition})
instead of 
$K$. 
Since 
$F=0$, 
\[
L_{X, T} = 
(G K_x^{-1} H)(X, T).
\]
Using
\beq
(G)_{ij} &=&
\cases{
-1 & ($i \in N, j \in I, i \sim j$) \cr
0 & (otherwise) 
}
\\
(H)_{ij} &=&
\cases{
-1 & ($i \in I, j \in N, i \sim j$) \cr
0 & (otherwise) 
}
\eeq
and the fact 
$t$
is the only vertex connected to the terminal 
$T$, 
we have 
\[
L_{X, T} = 
(K_x)^{-1}(x, t).
\]
Substituting it to 
Theorem \ref{bipartite}
yields the conclusion. 
\QED
\end{proof}
\noindent
{\it Proof of Theorem \ref{one impurity} }
We show that 
the dependence of 
$L_{X, T} = (K_x)^{-1}(x,t)$
on 
$x$
cancels with that of 
$\det K_x$.  
In fact, by taking the 
$(x,x)$-component of the resolvent equation 
\begin{equation}
K^{-1}_x  - K^{-1}  = K^{-1}_x (K - K_x) K^{-1} 
\label{resolvent}
\end{equation}
we have 
\beq
(K^{-1}_x)(x,x)  - (K^{-1})(x,x)
&=&
- (K^{-1}_x)(x,x)  (K^{-1})(x,x)
\eeq
thus 
\begin{equation}
(K_x)^{-1}(x,x)
=
\frac {
(K)^{-1}(x,x)
}
{
1 + (K)^{-1}(x,x)
}.
\label{diagonal}
\end{equation}
Moreover taking the 
$(x,t)$-component of 
(\ref{resolvent})
and using 
(\ref{diagonal}), 
\begin{equation}
(K_x)^{-1}(x,t)
=
\frac {
(K)^{-1}(x,t)}
{
1 + (K)^{-1}(x,x)
}.
\label{non-diagonal}
\end{equation}
Let 
$k_x$
be the submatrix of 
$K_x$ 
by eliminating the 
$x$-th row and $x$-th column.
By Cramer's formula, 
\[
(K)^{-1}(x,x)
=
\frac {\det (k^x)}{\det (K)}.
\]
On the other hand, 
expanding the determinant yields 
\[
\det (K_x) = \det (K) + \det (k^x)
\]
leading to 
\begin{equation}
1 + (K)^{-1}(x,x)
=
\frac {
\det (K) + \det (k^x)
}
{
\det (K)}
=
\frac {\det (K_x)}{\det (K)}.
\label{determinant}
\end{equation}
It then suffices to substitute 
(\ref{non-diagonal}), (\ref{determinant})
to 
Lemma \ref{one impurity1}
to finish the proof of 
Theorem \ref{one impurity}.
\QED

\subsection{Two impurity case I：
alternative proof of Theorem \ref{two impurities}
}
As we did in Section 3.3, 
we construct a circular planar graph 
$G_C$
as follows(Figure \ref{GC-Two-Impurities}) : 
(i)
add all the boundary vertices to $G_1$, 
(ii)
Let 
$A$, $B$ 
be the boundary vertices attached to 
$a$
and 
$b$
respectively, and identify all boundary vertices lying between 
$A$
and 
$B$, 
which we call 
$C$, and 
(iii)
identify the other boundary vertices in their respective region. 
Let 
$D, E, F, G$ 
be these vertices and set 
\beq
V(G_C) &:=& I \cup N, 
\quad
I := V(G_1), 
\quad
N := \{ T_1, T_2, T_3, A, B, C, D, E, F, G \}.
\eeq
%

\begin{figure}
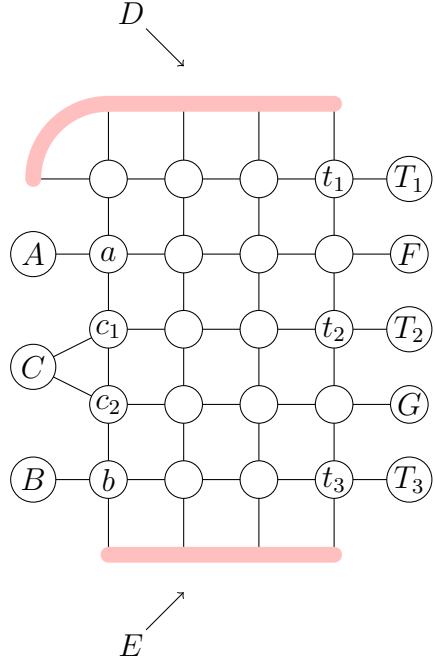

\begin{center}
\bet

\draw (0,0) grid +(3,4) ; 

\foreach \x in {0, 1, 2, 3}
{
\draw (\x, 0) --+(0, -1) ; 
\draw (\x, 4) --+(0, 1) ; 
}

\draw (0, 0) --+(-1, 0) ; 

\draw (0, 1) --+(-1, 0.5) ; 
\draw (0, 2) --+(-1, -0.5) ; 

\draw (0, 3) --+(-1, 0) ; 
\draw (0, 4) --+(-1, 0) ; 

\foreach \y in {0, 1, 2, 3, 4}
{
\draw (3, \y) --+(1, 0) ; 
}

\foreach \x in {0, 1, 2, 3, 4}
\foreach \y in {0, 1, 2, 3, 4}
{
\draw [ fill = white ] (\x, \y) circle (0. 25 cm) ; 
}
\draw (3, 0) node {$t_3$} ; 
\draw (3, 2) node {$t_2$} ; 
\draw (3, 4) node {$t_1$} ; 
\draw (0, 0) node {$b$} ; 
\draw (0, 1) node {$c_2$} ; 
\draw (0, 2) node {$c_1$} ; 
\draw (0, 3) node {$a$} ; 

\foreach \y in {0, 1.5, 3}
{
\draw [ fill = white ] (-1, \y) circle (0.3cm ) ; 
}
\draw (-1, 3) node {$A$} ; 
\draw (-1, 1.5) node {$C$} ; 
\draw (-1, 0) node {$B$} ; 

\foreach \y in {0, 2, 4}
{
\draw [ fill = white ] (4, \y) circle (0.3cm ) ; 
}
\draw (4, 0) node {$T_3$} ;
\draw (4, 2) node {$T_2$} ;
\draw (4, 4) node {$T_1$} ;

\draw [ line cap = round, color = pink, line width = 6 ] (3, 5) --+(-3, 0) ; 
\draw [ line cap = round, color = pink, line width = 6 ] (0, 5) arc (90 : 180 : 	1) ; 
\draw [ <- ] (1, 5.5)--+(-0.5, 0.5) ; 
\draw (0.3, 6.2) node {$D$} ; 
 
 \draw [ line cap = round, color = pink, line width = 6 ] (0, -1) --+(3, 0) ; 
\draw [ <- ] (1, -1.5) --+(-0.5, -0.5) ; 
\draw (0.3, -2.2) node {$E$} ; 
 
 \draw (4, 3) node {$F$} ; 

 \draw (4, 1) node {$G$} ; 
 
\ent
\end{center}
\caption{
$G_C$ in two impurity case
}
\label{GC-Two-Impurities}
\end{figure}
%


%
Let 
$t_1, t_2, t_3, a, b \in I$
be the vertices connected to the nodes 
$T_1, T_2, T_3, A, B$
respectively. 
By 
(\ref{forest}), 
$M(a,b)$
is equal to the number of the groves on 
$G_C$ 
such that 
$A \leftrightarrow T_1, \partial C \leftrightarrow T_2, B \leftrightarrow T_3$
and 
$D, E, F, G$
are isolated. 
The corresponding partition 
$\sigma \in {\cal P}_{pl}$
is bipartite (Figure \ref{Sigma}) so that by 
Theorem \ref{bipartite}
we have
\begin{equation}
Z_{G_C}(\sigma)
= 
\left|
\det 
\left(
\begin{array}{ccc}
L_{A,T_1} & L_{A, T_2} & L_{A, T_3} \\
L_{ C, T_1} & L_{ C, T_2} & L_{ C, T_3} \\
L_{B, T_1} & L_{B, T_2} & L_{B, T_3} \\
\end{array}
\right)
\cdot
\det(K)
\right|.
\label{zero}
\end{equation}
%

\begin{figure}
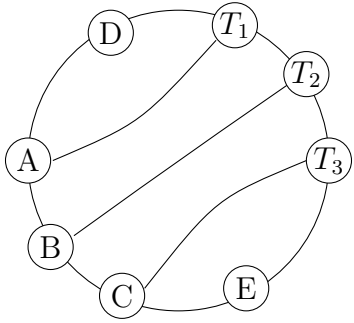

\begin{center}
\bet

\draw (0,0) circle (2cm) ; 

\draw [ fill = white ] (2, 0) circle (0.3 cm) ; 
\draw  (2, 0) node {$T_3$} ; 

\draw [ fill = white ] (1.7, 1.15) circle (0.3 cm) ; 
\draw  (1.7, 1.15) node {$T_2$} ; 

\draw [ fill = white ] (0.75, 1.8) circle (0.3 cm) ; 
\draw  (0.75, 1.8) node {$T_1$} ; 

\draw [ fill = white ] (-0.9, 1.7) circle (0.3 cm) ; 
\draw  (-0.9, 1.7) node {D} ; 

\draw [ fill = white ] (0.9, -1.7) circle (0.3 cm) ; 
\draw   (0.9, -1.7)  node {E} ; 

\draw [ fill = white ] (-2, 0) circle (0.3 cm) ; 
\draw  (-2, 0) node {A} ; 

\draw [ fill = white ] (-1.7, -1.15) circle (0.3 cm) ; 
\draw  (-1.7, -1.15) node {B} ; 

\draw [ fill = white ] (-0.75, -1.8) circle (0.3 cm) ; 
\draw  (-0.75, -1.8) node {C} ; 

\draw (-1.67, 0) .. controls (-0.5,0.5) .. (0.5, 1.6) ; 
\draw (-1.39, -1) .. controls (0, 0) .. (1.44, 1) ; 
\draw (-0.45, -1.7) .. controls (0.5, -0.5) .. (1.7, 0) ; 

\ent
\end{center}
\caption{
The corresponding planar partition
$\sigma$.
P, Q, R are red, A, B, C are blue, and D, E are split into red and blue.
}
\label{Sigma}
\end{figure}
%


%
It suffices 
to compute these matrix elements of 
$L$. 
Under the same notation in 
(\ref{decomposition}), we have 
$L = - F + G K^{-1} H = G K^{-1} H$. 
Hence 
\begin{equation}
L_{A, T_j}
=
(G K^{-1} H)_{A, T_j}
=
(K^{-1})_{a, t_j}, 
\quad
j = 1, 2, 3
\label{one}
\end{equation}
likewise for 
$L_{B, T_j}$. 
Similarly 
\begin{equation}
L_{C, T_i}
=
(G K^{-1} H)_{C, T_i}
=
\sum_{z \in \partial C}
(K^{-1})_{z, t_i}.
\label{two}
\end{equation}
By substituting (\ref{one}), (\ref{two})
into 
(\ref{zero}), 
we complete the proof of 
Theorem \ref{two impurities}.
\QED
%
%
\begin{remark}
If an impurity lies inside 
$G_1$, 
we generally have a non-planar partition 
$\tau$
with 
$Z(\tau) \ne 0$
so that 
the above argument is no longer valid.
\end{remark}
\subsection{Two impurity case II : impurities lying ``near" the boundary
}
In 
$G^{(2)}$, 
we consider the case in which both impurities satisfy 
$I^{(j)}_1 \in \partial G_1$, 
$j=1,2$, 
but 
$I^{(1)}_2 \notin \partial G_2$
(Figure \ref{Inner-Impurity}). 
By Theorem \ref{bijection}, 
there are two types of corresponding spanning trees : one is given in 
Figure \ref{A},  
whose contribution has been computed in Theorem \ref{two impurities}, 
and the other one in Figure \ref{B}.
As is explained in Figure \ref{One-Impurity}, 
the tree connecting 
$T_1$
and 
$T_2$
in Figure \ref{B} can be regarded as a composition of a TI-tree and  a TO-tree. 
In this subsection
we compute the contribution from the spanning trees in Figure \ref{B}.
Let 
$c \in G_1$
be the vertex connected to 
$I^{(1)}_2$
through a diagonal edge(Figure \ref{B}). 
%

\begin{figure}
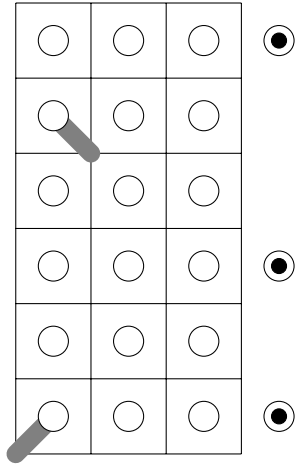

\begin{center}
\bet

\draw (0, 0) grid +(3, 6) ; 

\begin{scope} [ xshift = 0.5 cm, yshift = 0.5 cm ] 
{

\draw [ line cap = round, color = gray, line width = 7 ] (0, 4) --+(0.5, -0.5) ; 
\draw [ line cap = round, color = gray, line width = 7 ] (0, 0) --+(-0.5, -0.5) ;

\foreach \x in {0, 1, 2}
\foreach \y in {0, 1, 2, 3, 4, 5}
{
\draw [ fill = white ] (\x, \y) circle (0.2 cm) ; 
}


\foreach \x in {0, 2, 5}
{
\draw [ fill = white ] (3, \x) circle (0.2 cm) ;
\draw [ fill = black ] (3, \x) circle (0.1 cm) ; 
}

}
\end{scope}

\ent
\end{center}
\caption{
If 
$I^{(1)}_2 \notin \partial G_2$, 
we have extra spanning trees. 
}
\label{Inner-Impurity}
\end{figure}
%

\begin{figure}
\begin{center}
\bet
\draw (0, 0) grid +(3, 6) ; 

\begin{scope} [ xshift = 0.5 cm, yshift = 0.5 cm ] 
{
\draw [ line cap = round, color = red, line width = 5 ] (0, 4) --++ (2, 0) --++(0, 1) --++(1.5, 0) ; 
\draw [ line cap = round, color = blue, line width = 5 ] (-1, 3) --++(2, 0) --++(0, -1) --++(2, 0) ; 
\draw [ line cap = round, color = red, line width = 5 ] (0, 0)  --++(3.5, 0) ; 

\draw [ line cap = round, color = gray, line width = 7 ] (0, 4) --+(0.5, -0.5) ; 
\draw [ line cap = round, color = gray, line width = 7 ] (0, 0) --+(-0.5, -0.5) ;

\foreach \x in {0, 1, 2}
\foreach \y in {0, 1, 2, 3, 4, 5}
{
\draw [ fill = white ] (\x, \y) circle (0.2 cm) ; 
}


\foreach \x in {0, 2, 5}
{
\draw [ fill = white ] (3, \x) circle (0.2 cm) ; 
\draw [ fill = black ] (3, \x) circle (0.1 cm) ; 
}

}
\end{scope}

\ent
\end{center}
\caption{
A spanning trees contributing to $A(a,b)$. 
IO-trees are omitted. 
}
\label{A}
\end{figure}
%


\begin{figure}
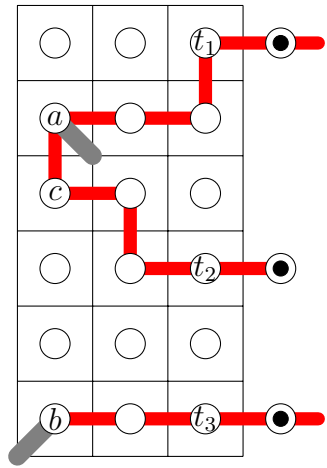

\begin{center}
\bet

\draw (0, 0) grid +(3, 6) ; 

\begin{scope} [ xshift = 0.5 cm, yshift = 0.5 cm ] 
{
\draw [ line cap = round, color = red, line width = 5 ] (3.5, 5) --++(-1.5, 0) --++(0, -1) --++(-2, 0) --++(0, -1) --++(1, 0) --++(0, -1) --++(2, 0) ; 
\draw [ line cap = round, color = red, line width = 5 ] 
(0, 0)  --++(3.5, 0) ; 

\draw [ line cap = round, color = gray, line width = 7 ] (0, 4) --+(0.5, -0.5) ; 
\draw [ line cap = round, color = gray, line width = 7 ] (0, 0) --+(-0.5, -0.5) ;

\foreach \x in {0, 1, 2}
\foreach \y in {0, 1, 2, 3, 4, 5}
{
\draw [ fill = white ] (\x, \y) circle (0.2 cm) ; 
}

\draw (0, 4) node {$a$} ; 
\draw (0, 3) node {$c$} ; 
\draw (0, 0 ) node {$b$} ; 

\foreach \x in {0, 2, 5}
{
\draw [ fill = white ] (3, \x) circle (0.2 cm) ; 
\draw [ fill = black ] (3, \x) circle (0.1 cm) ; 
}

\draw (2, 5) node {$t_1$} ; 
\draw (2, 2) node {$t_2$} ; 
\draw (2, 0) node {$t_3$} ; 

}
\end{scope}

\ent
\end{center}
\caption{
A spanning trees contributing to $B(a,b)$. 
IO-trees are omitted. 
}
\label{B}
\end{figure}
%


%
\begin{theorem}
\label{inner}
Let 
$a, b \in \partial G_1$.
The number 
$M(a,b)$ 
of perfect matchings satisfying 
$I^{(1)}_1=a$, 
$I^{(2)}_1=b$
and 
$I^{(1)}_2 \notin \partial G_2$, 
$I^{(2)}_2 \in \partial G_2$
is equal to
\beq
M(a,b) &=& A(a, b) + B(a,b)
\\
B(a,b) &:=& 
\left|
\det
\left( \begin{array}{ccc}
L_{a, t_1} & L_{a, t_2} & L_{a, t_3} \\
L_{b, t_1} & L_{b, t_2} & L_{b, t_3} \\
L_{c, t_1} & L_{c, t_2} & L_{c, t_3} \\
\end{array} \right)
\cdot
\det (K)
\right|
\eeq
where 
$A(a,b)$
is the one given in 
Theorem \ref{two impurities}
and 
\[
L_{x,y} :=(K^{-1})(x,y).
\]
\end{theorem}
It is also possible
to deal with the case where 
$I_2^{(j)} \notin \partial G_2$ 
$j=1,2$, 
with some extra terms computed similarly. 
\begin{proof}
We compute 
$B(a,b)$
which is the number of the spanning trees on 
$G_{1, T, R}$ 
corresponding to the ones in Figure \ref{B}.
As is done in subsection 3.3, 
we add all the boundary vertices to 
$G_1$, 
and let 
$A, B, C$
be the boundary vertices connected to 
$a,b, c$
respectively. 
We identify 
all boundary vertices between 
$C$
and
$B$,
which we call 
$D$(Figure \ref{GCB}). 
Similarly, 
we define vertices 
$E, F, G, H$.
Let 
\[
V(G_C) := I \cup N, 
\quad
I = V(G_1), 
\quad
N := \{ T_1, T_2, T_3, A, B,C, D, E, F, G, H \}.
\]
$B(a,b)$
is equal to the number of groves such that 
$T_1 \leftrightarrow A$, 
$T_2 \leftrightarrow C$, 
$T_3 \leftrightarrow B$, 
and the other nodes are isolated. 
Since 
the corresponding  
$\sigma \in {\cal P}_{pl}$
is bipartite, we can apply 
Theorem \ref{bipartite}. 
\QED
\end{proof}
%

\begin{figure}
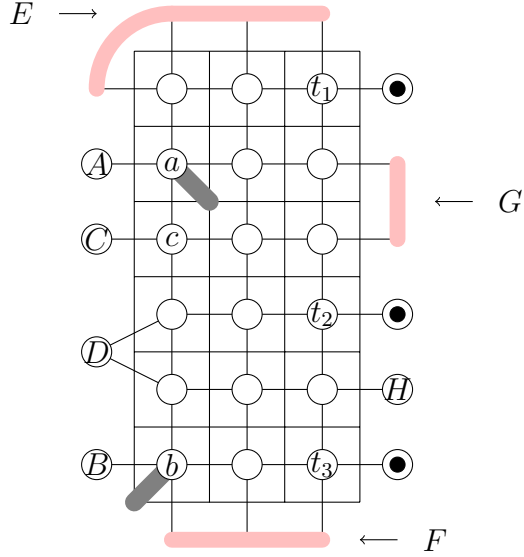

\begin{center}
\bet

\draw (0, 0) grid +(3, 6) ; 

\begin{scope} [ xshift = 0.5 cm, yshift = 0.5 cm ] 
{

\draw [ line cap = round, color = gray, line width = 7 ] (0, 4) --+(0.5, -0.5) ; 
\draw [ line cap = round, color = gray, line width = 7 ] (0, 0) --+(-0.5, -0.5) ; 

\foreach \y in {0,  3, 4, 5}
{
\draw (-1, \y) --+(4, 0) ;
}
\foreach \y in {1, 2}
{
\draw (0, \y) --+(3, 0) ; 
}
\foreach \x in {0, 1, 2}
{ 
\draw (\x, -1) --+(0, 7) ; 
}
\draw (0, 1) --+(-1, 0.5) ; 
\draw (0, 2) --+(-1, -0.5) ; 

\foreach \x in {0, 1, 2}
\foreach \y in {0, 1, 2, 3, 4, 5}
{
\draw [ fill = white ] (\x, \y) circle (0.2 cm) ; 
}

\foreach \y in {0, 1.5, 3, 4}
{
\draw [ fill = white ] (-1, \y) circle (0.2cm) ; 
}

\foreach \x in {0, 2, 5}
{
\draw [ fill = white ] (3, \x) circle (0.2 cm) ; 
\draw [ fill = black ] (3, \x) circle (0.1 cm) ; 
}

\draw (0, 4) node {$a$} ; 
\draw (0, 3) node {$c$} ; 
\draw (0, 0 ) node {$b$} ; 

\draw (-1, 4) node {$A$} ; 
\draw (-1, 3) node {$C$} ; 
\draw (-1, 0 ) node {$B$} ; 
\draw (-1, 1.5 ) node {$D$} ; 

\draw (2, 5) node {$t_1$} ; 
\draw (2, 2) node {$t_2$} ; 
\draw (2, 0) node {$t_3$} ;

\draw [ line cap = round, color = pink, line width = 6 ] (2, 6) --+(-2, 0) ; 
\draw [ line cap = round, color = pink, line width = 6 ] (0, 6) arc (90 : 180 : 1) ; 
\draw [ <- ] (-1, 6) --+(-0.5, 0) ; 
\draw (-2, 6) node {$E$} ; 

\draw [ line cap = round, color = pink, line width = 6 ] (0, -1) --+(2, 0) ; 
\draw [ -> ] (3, -1) --+(-0.5, 0) ; 
\draw (3.5, -1) node {$F$} ;

\draw [ line cap = round, color = pink, line width = 6 ] (3, 3) --+(0, 1) ; 
\draw [ <- ] (3.5, 3.5) --+(0.5, 0) ; 
\draw (4.5, 3.5) node {$G$} ; 

\draw [ fill = white ] (3, 1) circle (0.2 cm) ; 
\draw (3,1) node {$H$} ; 

}
\end{scope}

\ent
\end{center}
\caption{
$G_C$
to compute 
$B(a,b)$. 
}
\label{GCB}
\end{figure}
%


%
%
\subsection{More impurities case}
It is possible to extend our argument in 
Section 3.3 to the case where the number of impurities satisfy
$k \ge 3$, 
provided all impurities lie on the boundary. 
Let 
$(I^{(1)}_1, I^{(2)}_1, \cdots, I^{(k)}_1)
=
(a_1, a_2, \cdots, a_k)$, 
$a_j \in \partial G_1$, 
$j=1, 2, \cdots, k$
be the position of impurities.
Let 
$\partial C_j$, $j=1, 2, \cdots, k-1$
be the set of vertices in 
$\partial G_1$
lying between 
$a_j$
and 
$a_{j+1}$. 
\begin{theorem}
\label{many impurities}
Let 
$k \ge 2$.
If all impurities lie on the boundary and satisfy
$I^{(j)}_1=a_j$, 
$a_j \in \partial G_1$, 
$I^{(j)}_2 \in \partial G_2$, 
$j=1, 2, \cdots, k$, 
then the number of corresponding perfect matchings is equal to 
\beq
M(a_1, a_2, \cdots, a_k)
=
\left|
\det 
\left( \begin{array}{cccc}
L_{a_1, t_1} & L_{a_1, t_2} & \cdots & L_{a_1, t_{2k-1}} 
\\
L_{a_2, t_1} & L_{a_2, t_2} & \cdots & L_{a_2, t_{2k-1}} 
\\
\cdots &&&\\
L_{a_k, t_1} & L_{a_k, t_2} & \cdots & L_{a_k, t_{2k-1}} 
\\
L_{\partial C_1, t_1} & L_{\partial C_1, t_2} & \cdots & L_{\partial C_1, t_{2k-1}} 
\\
\cdots &  &  & \\
L_{\partial C_{k-1}, t_1} & L_{\partial C_{k-1}, t_2} & \cdots & L_{\partial C_{k-1}, t_{2k-1}} 
\end{array} \right)
\cdot
\det (K)
\right|.
\eeq
\end{theorem}
We also have the analogue of 
Theorem \ref{inner}. 
%
\subsection{One-dimension case}
When 
$G_1$
is the one-dimensional chain, the impurities always lie on the boundary. 
Hence 
we can compute the number of perfect matchings for all configuration of impurities, by using the methods discussed in Sections 3.3, 3.4. 
For instance, 
we consider the case described in Figure \ref{1Dim-1}. 
There are  
the two types spanning trees on 
$G_{1, T, R}$ 
given in 
Figures \ref{1Dim-A}, \ref{1Dim-C}. 
For each case 
we construct the circular planar graph
$G_C$
as we did in Sections 3.3, 3.4. 
The number of spanning trees in Figure \ref{1Dim-A} is equal to 
$A(a,b)$ 
given in Section 3.3. 
The number of spanning trees in Figure \ref{1Dim-C} is equal to 
$C(a,b)$, 
which is the number of groves satisfying 
$T_1 \leftrightarrow T_2$, 
$T_3 \leftrightarrow b$, 
and the other nodes are isolated. 
%

\begin{figure}
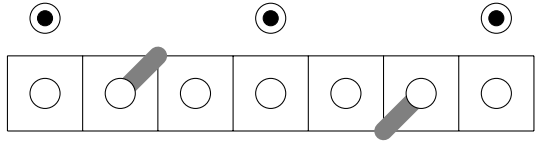

\begin{center}
\bet

\draw (0, 0) grid (7, 1) ; 

\begin{scope} [ xshift = 0.5 cm, yshift = 0.5 cm] 
{
\draw [ line cap = round, color = gray, line width = 7 ] (1, 0) --+(0.5, 0.5) ; 
\draw [ line cap = round, color = gray, line width = 7 ] (5, 0) --+(-0.5, -0.5) ; 

\foreach \x in {0, 1, 2, 3, 4, 5, 6}
{
\draw [ fill = white ] (\x, 0) circle ( 0.2 cm ) ; 
}

\foreach \x in {0, 3, 6}
{
\draw [ fill = white ] (\x, 1) circle ( 0.2 cm ) ; 
\draw [ fill = black ] (\x, 1) circle ( 0.1 cm ) ; 
}

}
\end{scope}

\ent
\end{center}
\caption{
An example impurity configuration in one-dimensional chain
}
\label{1Dim-1}
\end{figure}
%

\begin{figure}
\begin{center}
\bet

\draw (0, 0) grid (7, 1) ; 

\begin{scope} [ xshift = 0.5 cm, yshift = 0.5 cm] 
{
\draw [ line cap = round, color = red, line width = 5 ] (0, 1.5) --++(0, -1.5) --++(1, 0) ; 
\draw [ line cap = round, color = blue, line width = 5 ] (3, 1) --++(0, -1) --++(-1, 0) --++(0, -1) ; 
\draw [ line cap = round, color = red, line width = 5 ] (6, 1.5) --++(0, -1.5) --++(-2, 0) ;

\draw [ line cap = round, color = gray, line width = 7 ] (1, 0) --+(0.5, 0.5) ; 
\draw [ line cap = round, color = gray, line width = 7 ] (5, 0) --+(-0.5, -0.5) ; 

\foreach \x in {0, 1, 2, 3, 4, 5, 6}
{
\draw [ fill = white ] (\x, 0) circle ( 0.2 cm ) ; 
}

\foreach \x in {0, 3, 6}
{
\draw [ fill = white ] (\x, 1) circle ( 0.2 cm ) ; 
\draw [ fill = black ] (\x, 1) circle ( 0.1 cm ) ; 
}

}
\end{scope}

\ent
\end{center}
\caption{
A spanning tree contributing to $A(a,b)$
}
\label{1Dim-A}
\end{figure}
%

\begin{figure}
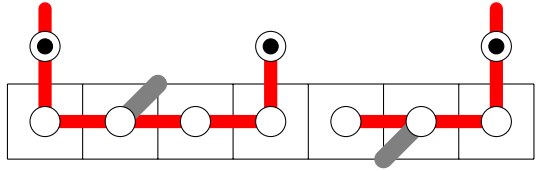

\begin{center}
\bet

\draw (0, 0) grid (7, 1) ; 

\begin{scope} [ xshift = 0.5 cm, yshift = 0.5 cm] 
{
\draw [ line cap = round, color = red, line width = 5 ] (0, 1.5) --++(0, -1.5) --++(3, 0) --++(0, 1) ; 
\draw [ line cap = round, color = red, line width = 5 ] (6, 1.5) --++(0, -1.5) --++(-2, 0) ;

\draw [ line cap = round, color = gray, line width = 7 ] (1, 0) --+(0.5, 0.5) ; 
\draw [ line cap = round, color = gray, line width = 7 ] (5, 0) --+(-0.5, -0.5) ; 

\foreach \x in {0, 1, 2, 3, 4, 5, 6}
{
\draw [ fill = white ] (\x, 0) circle ( 0.2 cm ) ; 
}

\foreach \x in {0, 3, 6}
{
\draw [ fill = white ] (\x, 1) circle ( 0.2 cm ) ; 
\draw [ fill = black ] (\x, 1) circle ( 0.1 cm ) ; 
}

}
\end{scope}

\ent
\end{center}
\caption{
A spanning tree contributing to $C(a,b)$
}
\label{1Dim-C}
\end{figure}
%

\section{Large scale limit}
In this section
we go back to the one impurity case and prove Theorems \ref{problem1}, \ref{problem2}.
We take 
$G_1$
to be the 
$n \times m$
rectangle : 
\[
G_1 = G_1^{(n, m)} :=
\{ (i, j) \, | \, 
1 \le i \le n, \; 1 \le j \le m \}.
\]
The terminal 
$T$
is attached to the vertex 
$(1,1)$. 
Let 
$l_T$
be the length of the TI-tree of the spanning tree 
$T$ 
given in Theorem \ref{bijection}, and let 
\begin{eqnarray*}
A^{(n,m)} 
&:=& (\triangle_{G_{1, T}} |_{G_1})^{-1}. 
\end{eqnarray*}
By 
Theorem \ref{one impurity}, 
\begin{eqnarray}
(1)&&\;
{\bf P}(
I_1 = (x,y)
)
=
\frac {
A^{(n,m)}( (x,y), (1,1) ) 
}
{
4 \sum_{x'=1}^n \sum_{y'=1}^m
A_{n, m}( (x',y'), (1,1) ) + 2
}
\label{probability}
\\
(2)&&\;
{\bf E}[l_T]
=\sum_{x'=1}^n \sum_{y'=1}^m
A^{(n,m)}( (x',y'), (1,1) ).
\label{length}
\end{eqnarray}
We would like to study the behavior of 
${\bf P}(I_1=(x,y))$
and
${\bf E}[l_T]$
as 
$n=m$ 
tends to infinity. 
\begin{proposition}
\label{G}
Let 
$n = m$.\\
(1)
For fixed 
$(x, y) \in G_1^{(n,m)}$, 
\beq
&&
\lim_{n \to \infty} A^{(n,n)}( (x,y), (1,1))
=
A(x,y)
\\
&&
A(x,y)
:=
\frac {2}{\pi^2}
\int \int_{(0, \pi)^2}
\frac {\sin (\theta x) \sin (\phi y) \sin \theta \sin \phi}
{2 -  \cos \theta -  \cos \phi}
d \theta d \phi.
\eeq
(2)
Under the uniform distribution on the spanning trees, 
\[
{\bf E}[l_T]
=
\sum_{(x,y) \in G_1}
A^{(n,n)}( (x,y), (1,1))
=
\frac {2}{\pi} \log n (1 + o(1)).
\]
as 
$n \to \infty$. 
\end{proposition}
\begin{proof}
(1)
The eigenvalues and the corresponding normalized eigenvectors of 
$A^{(n,m)}$
are given by 
\beq
e_{kl} &:=& 4 - 2 \cos 
\left(
\frac {k \pi}{n+1}
\right) - 2 \cos 
\left(
\frac {l \pi}{m+1}
\right), 
\\
\psi_{kl}(x,y) 
&:=&
\frac {2}{\sqrt{(n+1)(m+1)}}
\sin 
\left(
\frac {k \pi}{n+1} x
\right)
\sin 
\left(
\frac {l \pi}{m+1} y
\right), 
\eeq
$(k = 1, 2, \cdots, n, 
\;
l = 1, 2, \cdots, m)$. 
Hence
\begin{eqnarray*}
&&
A^{(n,m)}((x,y), (1,1))
\\
&=&
\frac {4}{(n+1)(m+1)}
\sum_{k=1}^n \sum_{l=1}^m
\frac {1}{e_{kl}}
\times
\\
&&
\times
\sin \left(\frac {k \pi}{n+1}x \right)
\sin \left(\frac {l \pi}{m+1}y\right)
\sin \left(\frac {k \pi}{n+1}\right)
\sin \left(\frac {l \pi}{m+1}\right)
\\
&\stackrel{n=m \to \infty}{\to}&
A(x,y).
\end{eqnarray*}
Here 
we note that the behavior of the integrand near the singularity 
$(\theta, \phi) = (0,0)$
is 
\begin{eqnarray*}
\frac {\sin (\theta x) \sin (\phi y) \sin \theta \sin \phi}
{4 - 2 \cos \theta - 2 \cos \phi}
&=&
\frac {\theta \cdot \phi}{\theta^2 + \phi^2}
\sin (\theta x) \cdot \sin (\phi y)
+
O(|\theta| + | \phi |)
\end{eqnarray*}
so that it is removable.
\\
%
(2)
The sum of 
$A^{(n,m)}((x,y), (1,1))$
is equal to 
\begin{eqnarray*}
&&
\sum_{x,y} A^{(n,m)}( (x,y), (1,1))
\\
&=&
\frac {4}{(n+1)(m+1)}
\sum_{k,\, l \mbox{:odd}}
\frac {1}{e_{kl}}
\cdot
\frac{
\sin^2 \left(\frac {k \pi}{n+1}\right)
\sin^2 \left(\frac {l \pi}{m+1} \right)
}
{
\left(
1 - \cos \frac {k \pi}{n+1}
\right)
\left(
1 - \cos \frac {l \pi}{m+1}
\right)
}.
\end{eqnarray*}
In what follows, we let 
$m=n$.
Due to the monotonicity of the integrand near the origin, we have, for a positive constant
$C$, 
\begin{eqnarray*}
&&\sum_{x,y} A_{n,n}( (x,y), (1,1))
\\
&&
\ge
\frac {1}{\pi^2}
\int_{\frac 1n}^{\pi} \int_{\frac 1n}^{\pi}
\frac {\sin^2 \theta \sin^2 \phi}
{(1 - \cos \theta)(1 - \cos \phi)}
\cdot
\frac {d \theta d \phi}
{4 - 2 \cos \theta - 2 \cos \phi}
-C.
\end{eqnarray*}
Since the integrand behaves 
\[
\frac {\theta^2 \phi^2}{\frac {\theta^2}{2} \cdot 
\frac {\phi^2}{2}}
\cdot
\frac {1}{\theta^2 + \phi^2}
=
\frac {4}{\theta^2 + \phi^2}
\]
near the origin, we have as 
$n \to \infty$ 
\begin{eqnarray*}
&&
\frac {1}{\pi^2}
\int_{\frac 1n}^{\pi} \int_{\frac 1n}^{\pi}
\frac {\sin^2 \theta \sin^2 \phi}
{(1 - \cos \theta)(1 - \cos \phi)}
\cdot
\frac {d \theta d \phi}
{4 - 2 \cos \theta - 2 \cos \phi}
\\
&&=
\frac {1}{\pi^2}
\int_{\frac 1n}^{\pi} dr \;
\frac 4r
\int_0^{\frac {\pi}{2}}
d \theta
(1 + o(1))
\\
&&
=
\frac {2}{\pi} \log n (1 + o(1)).
\end{eqnarray*}
On the other hand, 
\begin{eqnarray*}
&&\sum_{xy} G(x,y)
\le
\frac {1}{\pi^2}
\int_{\frac 1n}^{\pi} \int_{\frac 1n}^{\pi}
\frac {\sin^2 \theta \sin^2 \phi}
{(1 - \cos \theta)(1 - \cos \phi)}
\cdot
\frac {d \theta d \phi}
{4 - 2 \cos \theta - 2 \cos \phi}
\\
&&+
\frac {4}{(n+1)^2}
\sum_{l=1}^{\left[\frac {m+1}{2} \right]}
\frac {
\sin^2 \frac {\pi}{n+1} \sin^2 \frac {(2l-1)\pi}{m+1}
}
{
\left(
4 - 2 \cos \frac {\pi}{n+1} 
-2 \cos \frac {(2l-1)\pi}{m+1}
\right)
}
\frac {1}{
(1 - \cos \frac {\pi}{n+1})(1 - \cos \frac {(2l-1)\pi}{m+1})}
\\
&& +(n \leftrightarrow m,\; k \leftrightarrow l)
\\
&=:& I + II + III.
\end{eqnarray*}
The second term is estimated as  
\begin{eqnarray*}
II
&=&
\frac {4}{(n+1)^2}
\sum_l 
\frac {
(1 + \cos \frac {\pi}{n+1})(1 + \cos \frac {(2l-1)\pi}{m+1})
}
{4 - 2 \cos \frac {\pi}{n+1} - 2 \cos \frac {(2l-1)\pi}{m+1}}
\\
&\le&
(Const.)
\frac {4}{(n+1)^2}
\sum_l
\frac {(n+1)^2}{\pi^2 + (2l-1)^2 \pi^2}
=
O(1)
\end{eqnarray*}
and similarly for 
$III$.
Therefore 
the upper and lower bounds have the same leading order yielding 
\[
{\bf E}[l_T]
=
\sum_{x,y}A_{n,n}( (x,y), (1,1))
=
\frac {2}{\pi} \log n (1 + o(1)).
\]
\QED
\end{proof}
{\it Proof of Theorem \ref{problem1}}\\
(1) follows from 
\cite{NS1}, Example 3.6, and 
(2) follows from 
eq.(\ref{probability}) 
and 
Proposition \ref{G}(1), (2).
\QED
\\

\noindent
{\it Proof of Theorem \ref{problem2}}\\
Let 
$I_1 = {\bf r} \in G_1$
be the position of the impurity and take any 
$c > 0$.
By Theorem \ref{bijection}
if 
$| {\bf r} | \ge c L$
then 
$l_T \ge c L$
so that by Chebychev's inequality 
\beq
{\bf P}\left(
| {\bf r} | \ge c n
\right)
\le
{\bf P} \left(
l_T \ge c n
\right)
\le
\frac {1}{cn}
{\bf E}[ l_T ].
\eeq
For one-dimension, 
${\bf E}[ l_T ] < \infty$ 
by 
Theorem  \ref{problem1}(1)
and for two-dimension 
${\bf E}[ l_T ] = (const.) \log n (1+o(1))$, 
$n \to \infty$
by Proposition \ref{G}(2).
Thus
$\lim_{n \to \infty}
{\bf P}(| {\bf r} | \ge cn) = 0$. 
\QED

\vspace*{1em}
\noindent {\bf Acknowledgement }
This work is partially supported by 
JSPS grant Kiban-C no.22540140(F.N.) and no.23540155(T.S.).



\begin{thebibliography}{99}
%
\bibitem{Ciucu1}
Ciucu, M. : 
A random tiling model for two dimensional electrostatics. Mem. Amer. Math. Soc. {\bf 178}(2005), no. 839, p.1-106. 
%
\bibitem{Ciucu2}
Ciucu, M. : 
The scaling limit of the correlation of holes on the triangular lattice with periodic boundary conditions. (English summary) 
Mem. Amer. Math. Soc. {\bf 199}(2009), no. 935.
%
\bibitem{CK}
Ciucu, M. and Krattenthaler, C. : 
The interaction of a gap with a free boundary in a two dimensional 
dimer system, 
Commun. Math. Phys. {\bf 302} (2011) 253-289.
%
\bibitem{Fomin}
Fomin, S. :
Loop-erased walks and total positivity, 
Trans. Amer. Math. Soc. {\bf 353} No. 9(2001), 
p.3563-3583.
%
\bibitem{K}
Kenyon, R. : 
Lectures on Dimers, 
in Statistical Mechanics, 
Scott Sheffield, Thomas Spencer, Editors, 
IAS/Park city Mathematics Series, 
AMS. 
%
\bibitem{KW1}
Kenyon, Wilson : 
Boundary Partitions in Trees and Dimers, 
Trans. Amer. Math. Soc. {\bf 363} no.3 (2011), 1325-1364.
%
%
\bibitem{KW2}
Kenyon, Wilson : 
Combinatorics of Tripartite Boundary Connections 
for Trees and Dimers, 
Electron. J. Combin. {\bf 16}, no.1 (2009), 
%
\bibitem{NOS}
Fumihiko Nakano, Hirotaka Ono and Taizo Sadahiro 
``Connectedness of domino tilings with diagonal impurities", Discrete Mathematics, {\bf 310}(2010), p.1918-1931. 
%
\bibitem{NS1}
Fumihiko Nakano and Taizo Sadahiro
``Domino tilings with diagonal impurities", 
the special issue of Fundamenta Informaticae 
dedicated to the conference Lattice Path Combinatorics held in Siena. Vol. {\bf 117} (2012)  249-264. 
%
arXiv:0901.4824, 
%
\bibitem{NS2}
Fumihiko Nakano and Taizo Sadahiro  
``A bijection theorem for domino tiling with diagonal impurities", 
Journal of Statistical Physics, Vol. {\bf 139}, No. 4(2010), p.565-597.
arXiv : 0907.3252(math.CO, math.PR)
%
\bibitem{Wilson}
Wilson, D., 
Generating random spanning trees more quickly than the cover time, 
Proceedings of the 28th ACM on the Theory of Computing, 
ACM, New York, 1996, pp.296-303. 
%
\end{thebibliography}
\end{document}